\documentclass[12pt]{amsart}
\usepackage{amsthm}
\newtheorem{theorem}{Theorem}[section]
\newtheorem{lemma}[theorem]{Lemma}
\newtheorem{proposition}[theorem]{Proposition}
\theoremstyle{definition}
\newtheorem{definition}[theorem]{Definition}

\theoremstyle{remark}
\newtheorem{remark}[theorem]{Remark}

\counterwithin*{section}{part}

\usepackage{amssymb}
\usepackage{empheq}
\usepackage{stmaryrd}
\usepackage{enumerate}
\usepackage[shortlabels]{enumitem}
\usepackage{calc}
\usepackage{url}

\usepackage{mathabx}

\usepackage{xcolor}

\newcommand{\norm}[1]{\left\lVert#1\right\rVert}
\usepackage{mathtools}

\def\PP{\mathbb{P}}
\def\NN{\mathbb{N}}

\def\RR{\mathbb{R}}

\newcommand{\WEPAomega}{\textsf{WE}\mbox{-}\textsf{PA}^\omega}

\newcommand{\QFAC}{\textsf{QF}\mbox{-}\textsf{AC}}
\newcommand{\QFER}{\textsf{QF}\mbox{-}\textsf{ER}}
\newcommand{\DC}{\textsf{DC}}
\newcommand{\AC}{\textsf{AC}}

\newcommand{\INT}[1]{\int{#1}\,\mathrm{d}\PP}

\usepackage[left=2.0cm,%
                right=2.0cm,%
                top=2.5cm,%
                bottom=3.5cm,%
                headheight=12pt,%
                a4paper]{geometry}%

\begin{document}

\title{Proof mining and probability theory}

\author[Morenikeji Neri and Nicholas Pischke]{Morenikeji Neri${}^{\MakeLowercase a}$ and Nicholas Pischke${}^{\MakeLowercase b}$}
\date{\today}
\maketitle
\vspace*{-5mm}
\begin{center}
{\scriptsize 
${}^a$ Department of Computer Science, University of Bath,\\
Claverton Down, Bath, BA2 7AY, United Kingdom,\\
${}^b$ Department of Mathematics, Technische Universit\"at Darmstadt,\\
Schlossgartenstra\ss{}e 7, 64289 Darmstadt, Germany,\\ 
E-mails: mn728@bath.ac.uk, pischke@mathematik.tu-darmstadt.de}
\end{center}

\maketitle
\begin{abstract}
We extend the theoretical framework of proof mining by establishing general logical metatheorems that allow for the extraction of the computational content of theorems with prima facie ``non-computational'' proofs from probability theory, thereby unlocking a major branch of mathematics as a new area of application for these methods. Concretely, we devise proof-theoretically tame logical systems that, for one, allow for the formalization of proofs involving algebras of sets together with probability contents as well as associated Lebesgue integrals on them and which, for another, are amenable to proof-theoretic metatheorems in the style of proof mining that guarantee the extractability of effective and tame bounds from large classes of ineffective existence proofs in probability theory. Moreover, these extractable bounds are guaranteed to be highly uniform in the sense that they will be independent of all parameters relating to the underlying probability space, particularly regarding events or measures of them. As such, these results, in particular, provide the first logical explanation for the success and the observed uniformities of the previous ad hoc case studies of proof mining in these areas and further illustrate their extent. Beyond these systems, we provide extensions for the proof-theoretically tame treatment of $\sigma$-algebras and associated probability measures using an intensional approach to infinite unions. Lastly, we establish a general proof-theoretic transfer principle that allows for the lift of quantitative information on a relationship between different modes of convergence for sequences of real numbers to sequences of random variables.
\end{abstract}

\section{Introduction}

One of the fundamental driving questions of proof theory is the following: What is the computational content of a mathematical theorem and how can it be exhibited? Proof mining, which emerged as a subfield of mathematical logic in the 1990s through the work of Ulrich Kohlenbach and his collaborators,\footnote{Historically, proof mining has its roots in Kreisel's program of the ``unwinding of proofs'' \cite{Kre1951,Kre1952}. We refer to \cite{Koh2008} for a comprehensive monograph on proof mining and its applications until 2008 and we refer to the survey article \cite{KO2003} for details on the earlier development of proof mining.} aims at answering that question by extracting this computational content from theorems with proofs as they are found in the mainstream mathematical literature. This is, in particular, a non-trivial task as such proofs are prima facie noneffective, involving both classical logic as well as various non-computational (set-theoretic) principles. However, backed by a logical apparatus relying on the utilization of various methods from proof theory like functional interpretations and majorizability, this program of proof mining has had great success in various areas of mathematics, in particular regarding (nonlinear) analysis and optimization (see in particular the recent surveys \cite{Koh2017,Koh2019}).\\

Two areas where proof mining has previously only touched upon briefly are the fields of measure theory in general and probability theory in particular. Concretely, we refer to the works \cite{AO2021,ADR2012,AGT2010} which are essentially the only proof mining case studies in these areas preceding this work. From a practical perspective, this diffidence of proof mining regarding these areas is at least partially due to the fact that they are so far not substantiated by underlying logical methods as other areas of applications for proof mining are, a fact that also renders the previous applications, to a certain degree, ad hoc.

It is the aim of this paper to extend the current logical methods used in proof mining as to render them applicable to large classes of proofs from probability theory, in particular, so that they allow for a logical explanation of the success of the previous ad hoc case studies mentioned before (as well as of the various properties of the extracted content).\\

Concretely, the fundamental logical ``substrate'' of the proof mining program are the so-called general logical metatheorems.\footnote{Examples of such logical metatheorems for proof mining can, e.g., be found in \cite{FLP2019,GeK2008,GuK2016,Koh2005,Koh2008,KN2017,Leu2006,Leu2014,PS2022,Pis2023c,Pis2024,Pis2023,Sip2019}.} These use well-known so-called proof interpretations like G\"odel's functional interpretation \cite{Goe1958}, negative translations (see e.g.\ \cite{Kur1951}) and their extensions\footnote{For example, other proof interpretations used in proof mining are Kreisel's modified realizability interpretation \cite{Kre1951} for the treatment of semi-constructive proofs, monotone variants of the Dialectica interpretation \cite{Koh1996} to deal with the extraction of computable bounds from proofs or the related bounded functional interpretation \cite{FO2005}.} to provide a general result that quantifies and allows for the extraction of the computational content of large classes of theorems from their proofs (which in particular may involve classical logic and various non-computational principles). In that way, proof mining, as substantiated by these metatheorems, has led to hundreds of applications in the last decades. A crucial innovation in the techniques underlying these logical metatheorems was introduced by Kohlenbach in \cite{Koh2005}, marking the ``modern age'' of proof mining: while the metatheorems for proof mining preceding \cite{Koh2005} were based on ``pure'' systems for arithmetic in all finite types (see e.g.\ \cite{Koh2008,Tro1973}) and as such were restricted in their expressivity to dealing with spaces and structures that were representable as Polish metric spaces in the sense of Baire space (and thus separable), the paradigm first proposed in \cite{Koh2005} was to extend the language of the underlying systems with additional abstract base types which, together with additional constants and governing axioms, could be used to talk about much larger and broader classes of spaces beyond merely representable ones.\footnote{Whether a class of spaces can be treated in the context of the general logical metatheorems ultimately depends on the uniformity and complexity of the axioms describing said class as will also be discussed later in this paper.} The class of spaces and objects treated in this fashion has grown since then to a rather sizable amount, ranging from fundamental examples like general (non-separable) metric, hyperbolic, $\mathrm{CAT}(0)$, Banach and Hilbert spaces to much more involved objects like $\mathbb{R}$-trees, $L^p$-spaces, the dual of a (non-separable) Banach space as well as monotone operators and nonlinear semigroups, among many others. 

Further, the approach to represent various classes of spaces via abstract types has, in combination with an ingenious combination due to Kohlenbach (see the discussions in Section \ref{sec:metatheorem} later on for further details and references on this) of G\"odel's functional interpretation with (a suitable extension of) Howard's notion of majorizability \cite{How1973}, led to logical metatheorems that not only guarantee the extractability of effective bounds on non-effective existence statements but even guarantee a priori that these bounds are of a highly uniform nature, being independent of most data relating to any involved space or object.\\

This perspective of using abstract types to represent general spaces also underlies the approach towards a treatment of the various notions from probability theory proposed in this paper. Concretely, we utilize new abstract base types to introduce the relevant base sets for the measure spaces as well as to introduce the algebras of events over these sets on which the measures then operate. Using this approach, we establish a range of new metatheorems, the first of which allows for a treatment of proofs pertaining to the use of algebras (i.e.\ Boolean algebras of sets, also called fields) of events and so-called probability contents (i.e.\ finitely additive mappings, also called charges) on them. As e.g.\ highlighted in the seminal book by K.P.S.\ Bhaskara Rao and M.\ Bhaskara Rao \cite{RR1983} on general contents, these finitely additive measures offer a rich theory and are in a way ``more interesting, more difficult to handle, and perhaps more important than countably additive ones'' (which, as a statement, is attributed to Bochner in \cite{RR1983}). Further, beyond this mathematical perspective on the usefulness of (probability) contents, the nature of the applications of these systems discussed towards the end of this paper highlights that the theory of contents seems to be an invaluable perspective for applied proof theory in the context of probability theory which, from a logical point of view, further substantiates the statement quoted above and lets us assume that already this system for probability contents will provide a suitable base for proof mining developments in the context of probability theory in the future.

Beyond the theory of probability contents, we then provide extensions of the above mentioned system by an intensional treatment of countably infinite unions to provide access to the theory of $\sigma$-algebras and associated $\sigma$-additive probability measures and we also utilize an intensional approach to the space of bounded and Borel-measurable functions to introduce the Lebesgue integral in these contexts.\\

One of the most crucial features of the new metatheorems presented in this paper is that they provide the first concrete logical explanation of the uniformities of extractable bounds observed in the previously mentioned proof mining case studies, which were found to be independent of the measure, the underlying set and the algebra. As in the case of the first modern metatheorems mentioned before, this relies on a specific extension of the notion of majorizability due to Howard, which utilizes the probability content (or measure) to provide a corresponding notion of majorizability for the new abstract types. In particular, we want to note that this is the first time in proof mining that an extension of Howard's majorizability notion to abstract spaces is utilized that does not rely on any metric structure of the underlying spaces.\\

Besides guaranteeing this high degree of uniformity of the extracted bounds, these metatheorems for probability theory further elucidate the extent of the phenomenon of so-called proof-theoretic tameness in the sense of Kohlenbach \cite{Koh2020} (see also \cite{Mac2005,Mac2011} for related discussions of such phenomena), i.e.\ they show that also in the context of probability theory, one observes the empirical fact that most proofs, although in principle being subject to well-known G\"odelian phenomena, nevertheless ``seem to be tame in the sense of allowing for the extraction of bounds of rather low complexity''. Concretely, a crucial aspect of the approaches to various notions from probability theory taken here is that they avoid the computational strength inherent in some of the fundamental objects involved in this field (like the strength of comprehension needed to deal with countably infinite unions or Borel measurable functions, among other things) via the use of intensional methods and in that way, the complexity of the extractable information depends on (and can be a priori bounded in terms of) the complexity of the principles used in the corresponding proof and is not artificially high due to the abstract use of such notions. In that way, while the main base system taken in this paper is one that contains large amounts of comprehension (to illustrate the potential strength of systems which are amenable to proof mining methods), the approach to the various objects from probability theory taken here does not rely on these strong principles at all and thus can also immediately be developed over suitably weak subsystems (as e.g.\ the collection of systems introduced in \cite{Koh1996b} based on the Grzegorczyk hierarchy \cite{Grz1953}) where most of the mathematics discussed here can be similarly carried out but where then bounds of correspondingly low complexity could be guaranteed and so the approach chosen here for these notions provides the right system to illustrate this apparent tameness of the area.\\

The present approach to proof mining and probability theory is in particular to be distinguished with previous work on logical aspects of (quantitative) probability theory. On the proof-theoretic side, the main preceding work is that by Kreuzer \cite{Kre2015} on extracting computational content from proofs in measure theory which however relies on strong forms of comprehension for treating measure spaces, resulting in a rather restricted formal theory with a more limited scope of analyzable theorems, and which further does not guarantee any of the uniformity features for the quantitative information extracted thereby in contrast to the present approach. Outside of proof theory, finitizations of concepts from probability theory that enjoy similar uniformities as the ones considered here have also been obtained using tools from model theory, particularly ultraproducts \cite{AI2013,DI2017,GT2014}, though differing from the results presented here in both method and scope. The first crucial difference is that the proof-theoretic metatheorems presented here actually provide a method for extracting the uniform quantitative information of which the model-theoretic approach only can infer the existence of. Beyond that, as already commented on before, in the present approach the complexity of that information can be gauged beforehand based on the principles used in the proof and the proof-theoretic tameness of that information can be guaranteed thereby. Lastly, the model-theoretic approach is essentially fixed to focus on convergence statements and uniformities relating to their so-called metastable formulations whereas the proof-theoretic approach presented here does not have that limitation (which has a crucial impact on the range of possible applications as will also be discussed again in the end of the introduction). These facts crucially separate our work from the model-theoretic approach. Nevertheless, the model-theoretic approaches are certainly not subsumed by our work but rather complementary. For one, the model-theoretic approaches to uniformities of rates of metastability in probability theory only rely on the truth of a statement while our proof-theoretic results rely on the provability of the underlying theorem in an (albeit very strong) underlying theory. Also, we want to highlight that in particular the approach by Due\~{n}ez and Iovino \cite{DI2017} seems to rely, upon a closer inspection, on rather similar ideas for approaching some of the initial objects in question (e.g.\ by effectively treating probability contents on algebras instead of measures on $\sigma$-algebras and treating these underlying spaces with their operations as abstract entities), a fact that illustrates the apparent similarity of the problems that both the proof-theoretic and the model-theoretic approaches to these types of questions face and which in particular further highlights the naturalness of the approach followed here. However, our approach starts to crucially differ also conceptually instead of just methodologically from the work in \cite{DI2017} in the way that infinite unions and integration are treated in our systems and in the way that the uniformities of associated (extracted) bounds are guaranteed.\\

Still, besides these logical considerations it is the applications of the present metatheorems to probability theory which are arguably the most important consequence of the present work and we hence use the last two sections of this paper to substantiate the applicability of the logical metatheorems introduced here.  

For one, we outline in detail how the quantitative analysis of Egorov's theorem as presented in the seminal case study for proof mining in measure theory from \cite{ADR2012} formalizes in our systems (and thereby, as mentioned before, we also provide the first logical explanation of the uniformities of the bounds extracted therein\footnote{This is in particular to be compared with the work by Due{\~n}ez and Iovino \cite{DI2017} where a model-theoretic account of the uniformities of a quantiative variant of the dominated convergence theorem, partially akin to the results of \cite{ADR2012}, is given. These considerations of \cite{DI2017} however avoid Egorov's theorem and with that the central object of study in \cite{ADR2012} by which their quantitative version of the dominated convergence theorem is established therein.}). In particular, connecting to the previous discussion on the theory of probability contents as a suitable base for a formal system for proof mining for general probability measures, the analysis of the results from \cite{ADR2012} provided later highlights that these remain true for probability contents and so, in particular, illustrate that the notions and proofs produced in the work \cite{ADR2012} by following the finitary perspective of proof mining provide the right notions to simultaneously extend the analyzed results also to the theory of contents. In that way, this points to the apparent empirical phenomenon that finitary quantitative variants of notions and results from the theory of probability measures, as suggested by proof mining, seem to provide suitable analogue notions for the theory of probability contents and it also highlights the naturalness of the theory of contents as an underlying medium for developing a logical account of probability theory in the sense of proof mining. In that way, the fundamental relevance of (probability) contents is essentially rediscovered by the proof-theoretic approach presented in this paper.

For another, we establish a general so-called proof-theoretic transfer principle that allows for a lift of computational information on the relation between modes of convergence of sequences of real numbers to sequences of random variables (thereby providing a formal footing for this type of strategy which is rather abundant in probability theory and in particular features in some forthcoming case studies on proof mining and probability theory by the first author).\\

Besides these examples of application discussed here, the applicability of the systems presented in this paper is further substantiated by the fact the they explain the previously mentioned application \cite{AO2021} preceding this work as well as the very recent works \cite{Ner2024,NP2024} (together with forthcoming work by the authors together with Thomas Powell) as instances of the present methodology. In particular, in the context of the works \cite{Ner2024,NP2024}, the extractive proof-theoretic perspective of this work was crucial for obtaining the respective results.\\

Beyond that, we think that the present work lends itself both to further theoretical investigations on the extension of proof mining methods to further notions from probability and measure theory like Bochner integrals and martingales, among many others, as well as to substantiate and carry out many further (and potentially much more sophisticated) applications of proof mining in this area beyond the already mentioned. In particular, we want to mention that most ideas developed here could be extended, mutatis mutandis, to general finite contents and measures.

\section{Preliminaries}

The basic system that we rely on is the system $\mathcal{A}^\omega=\WEPAomega+\QFAC+\DC$ for classical analysis in all finite types as commonly used in proof mining, formalized via (a weakly extensional variant of) Peano arithmetic in all finite types together with a few choice principles (see e.g.\ \cite{Koh2005} where this notation for the system was, presumably, first introduced). As all systems introduced here will be extensions of this system $\mathcal{A}^\omega$, we in this section sketch the essential features relevant for this paper. For any further details, we refer to the works \cite{Koh2008,Tro1973}.\\

Here, we follow the definition of weakly extensional Peano arithmetic in all finite types $\WEPAomega$ as e.g.\ given in \cite{Koh2008} (see also \cite{Tro1973}) and, in that way, we do not recall all the defining features $\WEPAomega$ here and only focus on the three main aspects which are relevant in detail for this paper. In general, we denote function types using the bracket notation used in \cite{Koh2008}, i.e.\ $\rho(\tau)$ is the type of functions that map objects of type $\tau$ to objects of type $\rho$, and we use $T$ to denote the set of all finite types as usual, i.e.
\[
0\in T,\quad \rho,\tau\in T\rightarrow \rho(\tau)\in T.
\]
As usual, we denote pure types by natural numbers by setting $n+1:=0(n)$. The four central properties of $\WEPAomega$ that we need here are that, for one, the only primitive relation is equality at type $0$ (denoted by $=_0$) and higher-type equality is only defined as an abbreviation via recursion with
\[
x^{\tau(\xi)}=_{\tau(\xi)}y^{\tau(\xi)}:=\forall z^{\xi}\left( xz=_\tau yz\right).
\]

For another, $\WEPAomega$ crucially does not contain the full extensionality principles
\[
\forall x^{\tau(\rho)},y^{\rho},{y'}^{\rho}\left(y=_{\rho}y'\to xy=_{\tau}xy'\right) \tag{$\mathsf{E}_{\rho,\tau}$}
\]
because this would not allow for a result on program extraction. Instead, it only contains the quantifier-free extensionality rule
\[
\frac{A_0\to s=_\rho t}{A_0\to r[s/x^\rho]=_\tau r[t/x^\rho]}\tag{$\QFER$}
\]
where $A_0$ is a quantifier-free formula, $s$ and $t$ are terms of type $\rho$ and $r$ is a term of type $\tau$.

Further, $\WEPAomega$ contains constants $\underline{R}_{\underline{\rho}}$ for simultaneous primitive recursion in the sense of G\"odel \cite{Goe1958} and Hilbert \cite{Hil1926} as governed by the axioms
\[
\begin{cases}
(R_i)_{\underline{\rho}}0\underline{y}\underline{z}=_{\rho_i}y_i\\
(R_i)_{\underline{\rho}}(Sx)\underline{y}\underline{z}=_{\rho_i}z_i(\underline{R}_{\underline{\rho}}x\underline{y}\underline{z})x
\end{cases}\text{ for }i=1,...,k
\tag{$\underline{R}$}
\]
where $\underline{\rho}=\rho_1,\dots,\rho_k$ is a tuple of types, $\underline{y}=y_1,\dots,y_k$ with $y_i$ of type $\rho_i$ and $\underline{z}=z_1,\dots,z_k$ with $z_i$ of type $\rho_i(0)\underline{\rho}^t$ where we write $\underline{\rho}^t:=(\rho_k)\dots(\rho_1)$.

Lastly, due to the inclusion of the combinators of Sch\"onfinkel \cite{Sch1924} in the language of $\WEPAomega$, the system allows the definition of $\lambda$-abstraction in the sense that for any term $t$ of type $\tau$ and any variable $x$ of type $\rho$, we can construct a term $\lambda x.t$ of type $\tau(\rho)$ such that the free variables of $\lambda x.t$ are exactly those of $t$ without $x$ and so that
\[
\WEPAomega\vdash (\lambda x.t)(s)=_\tau t[s/x]
\]
for any term $s$ of type $\rho$.\\

Next to $\WEPAomega$, we as usual define the principle of quantifier-free choice $\QFAC$, i.e.\footnote{Here, and in the following, we use the notation $\underline{Y}\underline{x}$ to abbreviate $Y_1\underline{x},\dots,Y_k\underline{x}$ for $\underline{Y}=Y_1,\dots,Y_k$.}
\[
\forall \underline{x}\exists\underline{y}A_0(\underline{x},\underline{y})\to \exists\underline{Y}\forall\underline{x}A_0(\underline{x},\underline{Y}\underline{x})\tag{$\mathsf{QF}$-$\mathsf{AC}$}
\]
with $A_0$ quantifier-free and where the types of the variable tuples $\underline{x}$, $\underline{y}$ are arbitrary, as well as the principle of dependent choice $\DC$ defined as the collection of $\DC^{\underline{\rho}}$ for all tuples of types $\underline{\rho}$ with
\[
\forall x^0,\underline{y}^{\underline{\rho}}\exists \underline{z}^{\underline{\rho}} A(x,\underline{y},\underline{z})\to \exists \underline{f}^{\underline{\rho}(0)}\forall x^0 A(x,\underline{f}(x),\underline{f}(S(x))) \tag{$\DC^{\underline{\rho}}$}
\]
where $\underline{f}^{\underline{\rho}(0)}$ stands for $f_1^{\rho_1(0)},\dots,f_k^{\rho_k(0)}$ and $A$ may now be arbitrary.\\

Over $\mathcal{A}^\omega$, we will have to rely on some chosen representation of the real numbers as a Polish space and for that we follow definitions and conventions given in \cite{Koh2008}. In particular, rational numbers are represented using pairs of natural numbers and, in that context, we fix the same paring function $j$ as in \cite{Koh2005}:
\[
j(n^0,m^0):=\begin{cases}\min u\leq_0(n+m)^2+3n+m[2u=_0(n+m)^2+3n+m]&\text{if existent},\\0^0&\text{otherwise}.\end{cases}
\]
The usual arithmetical operations $+_\mathbb{Q},\cdot_\mathbb{Q},\vert\cdot\vert_\mathbb{Q}$, etc., are then primitive recursively definable through terms that operate on such codes and the usual relations $=_\mathbb{Q}$, $<_\mathbb{Q}$, etc., are definable via quantifier-free formulas.

For real numbers we then rely on a representation via fast converging Cauchy sequences of rational numbers with a fixed Cauchy modulus $2^{-n}$ (see \cite{Koh2008} for details), i.e.\ via objects of type $1$, and we consider $\mathbb{N}$ and $\mathbb{Q}$ as being embedded in that representation via the constant sequences. Also here, the usual arithmetical operations like $+_\mathbb{R}$, $\cdot_\mathbb{R}$, $\vert\cdot\vert_\mathbb{R}$, etc., are primitive recursively definable through closed terms and the relations $=_\mathbb{R}$ and $<_\mathbb{R}$, etc., now operating on type $1$ objects, are representable via formulas of the underlying language. Naturally, these relations are not decidable anymore but are given by $\Pi^0_1$- and $\Sigma^0_1$-formulas, respectively.

In the context of this representation of reals, we will later rely on an operator $\widehat{\cdot}$ which allows for an implicit quantification over all such fast-converging Cauchy sequences of rationals. Following \cite{Koh2008}, we define this operator via
\[
\widehat{x}n:=\begin{cases}xn&\text{if }\forall k<_0n\left( \vert xk-_\mathbb{Q}x(k+1)\vert_\mathbb{Q}<_\mathbb{Q}2^{-k-1}\right),\\
xk&\text{for $k<_0n$ least with }\vert xk-_\mathbb{Q}x(k+1)\vert_\mathbb{Q}\geq_\mathbb{Q}2^{-k-1}\text{ otherwise},\end{cases}
\]
turning $x$ of type $1$ into a fast-converging Cauchy sequence $\widehat{x}$, and we refer to \cite{Koh2008} for any further discussions of its properties.\\

In the context of the bound extraction theorems later on, we will rely on a canonical selection of a Cauchy sequence representing a give real number. Naturally, such an association will be non-effective. However, it will suffice that the operation behaves well-enough w.r.t.\ the notion of majorization. Following \cite{Koh2005}, this can be achieved for non-negative numbers via the function $(\cdot)_\circ$ defined by
\[
(r)_\circ(n):=j(2k_0,2^{n+1}-1),\label{def:circDef}
\]
where
\[
k_0:=\max k\left[\frac{k}{2^{n+1}}\leq r\right].
\]
Later, we will need an extension of this function $(\cdot)_\circ$ to all real numbers such that we retain these nice properties regarding majorizability and so, for $r<0$, we consider $(r)_\circ$ to be defined by
\[
(r)_\circ(n)=j(2\bar{k}_0\dotdiv 1,2^{n+1}-1)
\]
where 
\[
\bar{k}_0:=\max k\left[\frac{k}{2^{n+1}}\leq \vert r\vert\right].
\]
Then $(r)_\circ(n)=-_\mathbb{Q}(\vert r\vert)_\circ(n)$ and we get the following lemma containing exactly the properties that we later need for this notion to be useful in the context of majorizability (extending Lemma 2.10 from \cite{Koh2005}):

\begin{lemma}[essentially {\cite[Lemma 2.10]{Koh2005}}, see also {\cite[Lemma 2.1]{Pis2023}}]\label{lem:circprop}
Let $r\in\mathbb{R}$. Then:
\begin{enumerate}
\item $(r)_\circ$ is a representation of $r$ in the sense of the above (see again e.g.\ \cite{Koh2008}).
\item For $s\in [0,\infty)$, if $\vert r\vert \leq s$, then $(r)_\circ\leq_1 (s)_\circ$.
\item $(r)_\circ$ is nondecreasing (as a type 1 function).
\end{enumerate} 
\end{lemma}

Lastly, we write $r_\alpha$ for the unique real represented by $\widehat{\alpha}$ for a given a sequence $\alpha\in\mathbb{N}^\mathbb{N}$ and we sometimes write $[\alpha](n)$ for the $n$-th element of that sequence for better readability.\\

In terms of notation, we want to note that, to enhance readability, we will omit the subscripts of the arithmetical operations for $\mathbb{R}$ everywhere and similarly, we will also omit types of variables whenever convenient and omit types in proofs almost always. Further, we will omit the operation $\cdot_{\mathbb{R}}$ often altogether. Lastly, we throughout denote the powerset of a set $X$ by $2^X$.

\section{Systems for algebras of sets}

In this section, we develop the underlying systems on which we will later bootstrap our treatment of probability contents and probability measures as well as of all the other respective extensions discussed before. As such, we begin with a treatment of algebras of sets as the most basic underlying algebraic notion that is essential for the theory of contents. For references on these basic definitions and their properties, if nothing else is mentioned otherwise, we mainly refer to \cite{RR1983}.

\begin{definition}[Algebra of sets]
Let $\Omega$ be a set and $S\subseteq 2^\Omega$. Then $S$ is called an algebra of sets (or simply an algebra) if $\emptyset\in S$ and for any $A,B\in S$, it holds that $A^c:=\Omega\setminus A\in S$ and $A\cup B\in S$.
\end{definition}

The approach that we take towards a formal system for algebras of sets is to use abstract types to represent both the underlying ground set $\Omega$ as well as the algebra $S\subseteq 2^\Omega$. One then has to restore the structure of $S$ as a collection of subsets over $\Omega$ with certain operations on them by including additional constants that reintroduce these operations in this abstract setting.\\

Concretely, to form a system for the treatment of algebras, we extend the previously discussed set of types $T$ by two new abstract types $\Omega$ and $S$, forming the extended set of types $T^{\Omega,S}$ defined by
\[
0,\Omega,S\in T^{\Omega,S},\quad \rho,\tau\in T^{\Omega,S}\rightarrow \rho(\tau)\in T^{\Omega,S},
\]
and, over the resulting language, we then utilize this augmented set of types to introduce the following new constants to induce the usual structure on the set represented by $S$ in relation to $\Omega$ as mentioned before:
\begin{itemize}
\item $\mathrm{eq}$ of type $0(\Omega)(\Omega)$;
\item $\in$ of type $0(S)(\Omega)$;
\item $\cup$ of type $S(S)(S)$;
\item $(\cdot)^c$ of type $S(S)$;
\item $\emptyset$ of type $S$;
\item $c_\Omega$ of type $\Omega$.
\end{itemize}
The constant $\mathrm{eq}$ serves as an abstract account of the equality relation between objects of type $\Omega$ while the constant $\in$ serves as an abstract account of the element relation between elements as objects of type $\Omega$ and sets as objects of type $S$. The constants $\cup$ and $(\cdot)^c$ reintroduce the respective operations of union and complement for the abstract type $S$ and $\emptyset$ provides a constant representing the empty set. The constant $c_\Omega$ in particular is intended to witness that the underlying set $\Omega$ is non-empty. We often simply write $A^c$ instead of $(A)^c$ for $A$ of type $S$. Further, we abbreviate $\in x A =_0 0$ by $x\in A$ and, similarly, we write $x\not\in A$ for $\in x A\neq_0 0$. Lastly, we define $\Omega:=\emptyset^c$ as a notation of the top element of $S$.\footnote{This is not to be confused with the type $\Omega$ but the context will make it clear which of these two readings is intended.} Also regarding notation, we introduce intersections as an abbreviation by defining
\[
A\cap B:=(A^c\cup B^c)^c
\]
for terms $A^S, B^S$.\\

We write $x=_\Omega y$ as an abbreviation of $\mathrm{eq} xy=_00$  for objects $x^\Omega,y^\Omega$. Using $\in$, we introduce equality on $S$ via the following abbreviation: for $A^S$ and $B^s$, we define 
\[
A =_S B:\equiv \forall x^\Omega(x\in A \leftrightarrow x \in B).
\]
Note that $=_S$ clearly is, probably, an equivalence relation. Furthermore, we introduce the abbreviation 
\[
A\subseteq_S B :\equiv \forall x^\Omega \left(x \in A \rightarrow x \in B\right)
\]
for $A,B$ of type $S$ and it is immediate to show that $\subset_S$ forms a partial order with respect to equality defined by $=_S$.\\

For axioms, we first specify that $\mathrm{eq}$ represents an equivalence relation:
\begin{align*}
&\forall x^\Omega,y^\Omega(\mathrm{eq} xy\le_0 1), \tag*{$(\mathrm{eq})_1$} \\
&\forall x^\Omega,y^\Omega,z^\Omega \left(x=_\Omega x\land \left( x=_\Omega y \to y=_\Omega x\right)\land \left(x=_\Omega y\land y=_\Omega z\to x=_\Omega z\right)\right). \tag*{$(\mathrm{eq})_2$}
\end{align*}
Further, we axiomatize that $\in$, as a relation, is bounded by $1$ on all inputs and behaves as an element relation regarding the operations of union and complement as well as with respect to the empty set:
\begin{align*}
&\forall x^\Omega \forall A^S (\in x A \le_0 1), \tag*{$(\in)_1$} \\
&\forall x^\Omega (x \not\in \emptyset), \tag*{$(\in)_2$}\\
&\forall x^\Omega \forall A^S, B^S(x \in A \cup B \leftrightarrow x \in A \lor x \in B), \tag*{$(\in)_3$}\\
&\forall x^\Omega \forall A^S (x \in A^c \leftrightarrow x \not\in A). \tag*{$(\in)_4$}
\end{align*}
Based on the fact that inclusions of elements $x^\Omega$ in elements $A^S$ as facilitated by $\in$ are quantifier-free assertions, the above axioms are (generalized) $\Pi_1$-sentences and so they are in particular immediately admissible in the context of bound extraction theorems based on the Dialectica interpretation (as will be discussed later in more detail).

\begin{definition}
We write $\mathcal{F}^\omega$ for the system resulting from $\mathcal{A}^\omega$ over the augmented language including the types $\Omega,S$ (where all the respective constants and axioms now are allowed to also refer to these new types, if applicable) by extending this system with the constants $\mathrm{eq},\in,\cup,(\cdot)^c,\emptyset,c_\Omega$ and the axioms $(\mathrm{eq})_1$ -- $(\mathrm{eq})_2$ as well as $(\in)_1$ -- $(\in)_4$.
\end{definition}

We now begin by showing some basic properties of the above operations on algebras provable in this system $\mathcal{F}^\omega$ which, for one, amount to deriving the essential algebraic properties of $S$ as a subalgebra of the full Boolean algebra of the power set of $\Omega$. Further, for another, all algebraic operations on $S$ behave in a provably extensional way.

\begin{proposition}\label{pro:basic_prop_fields}
The operations $\cup$ and $(\cdot)^c$ are provably extensional in $\mathcal{F}^\omega$, i.e.\ $\mathcal{F}^\omega$ proves:
\begin{enumerate}[(i)] 
\item $\forall A^S, {A'}^S, B^S, {B'}^S(A=_SA'\land B=_SB'\rightarrow A\cup B=_SA'\cup B')$,
\item $\forall A^S, {A'}^S(A=_SA'\rightarrow A^c=_S{A'}^c)$.
\end{enumerate}
Further, all the axioms of Boolean algebras, instantiated using $\cap,\cup,(\cdot)^c$ and $=_S$, can be derived in $\mathcal{F}^\omega$. Lastly, over $\mathcal{F}^\omega$, $A \subseteq_S B$ is equivalent to both $A =_S B \cap A$ and $B =_S A \cup B$ for terms $A^S,B^S$.
\end{proposition}
\begin{proof}
We only show items (i) and (ii) as these illustrate the style of proof that one typically follows in $\mathcal{F}^\omega$ to reason about the algebraic structure of $S$. The identities of Boolean algebras and the equivalent formulations of the order in terms of meet and join are then easily derived from the axioms $(\in)_2,\dots,(\in)_4$.
\begin{enumerate}[(i)]
\item Fix $A, A', B, B'$ and assume $A=_SA'$ as well as $B=_SB'$. We need to show that 
\[
\forall x(x\in A\cup B\leftrightarrow x\in A'\cup B'). 
\]
Let $x$ be arbitrary. Then $x\in A\cup B$ is equivalent to $x\in A\lor x\in B$ by $(\in)_3$. By assumption of $A=_SA'$ and $B=_SB'$, we have that $x\in A\lor x\in B$ is equivalent to $x\in A'\lor x\in B'$ and so to $x\in A'\cup B'$ by $(\in)_3$, which yields the claim.
\item Fix $A,A'$ and assume $A=_SA'$. We need to show that
\[
\forall x(x\in A^c\leftrightarrow x\in {A'}^c).
\]
Let $x$ be arbitrary. Then $x\in A^c$ is equivalent to $x \not\in A$ by $(\in)_4$. By assumption of $A=_SA'$, we have that $x\not\in A$ is equivalent to $x\not\in A'$ and thus to $x\in {A'}^c$ again by $(\in)_4$.
\end{enumerate}
\end{proof}

\begin{remark}
Also the constants $\mathrm{eq}$ and $\in$ are immediately provably extensional in $\mathcal{F}^\omega$, as can be shown using the quantifier-free extensionality rule.
\end{remark}

Using the recursor constants of the underlying language of $\mathcal{F}^\omega$ in combination with the union operation $\cup$ immediately allows one to also talk about arbitrary finite unions. Concretely, given a sequence of events $A^{S(0)}$ and two natural numbers $n^0\leq_0 m^0$, we use the abbreviation
\[
\bigcup_{i=n}^{m}A(i):=R_S(m-n,A(n),\lambda B,x.(B\cup A(n+x+1)))
\]
where $R_S$ is a (single) type $S$ recursor constant. For $m<_0n$, we simply set $\bigcup_{i=n}^{m}A(i):=\emptyset$. We then dually write 
\[
\bigcap_{i=n}^mA(i):=\left(\bigcup_{i=n}^m(A(i))^c\right)^c.
\]
It easy to show by induction that the previous extensionality result for $\cup$ extends to these finite unions.

\section{Systems for contents on algebras of sets}

We now augment the previous system $\mathcal{F}^\omega$ for the treatment of algebras so that we arrive at a system suitable for treating proofs from the theory of probability contents in the sense of the following definition:\footnote{While contents are also studied over much sparser structures than algebras of sets, we here only consider the case of a content defined on such an algebra.}

\begin{definition}[Contents]
Let $\Omega$ be a set and $S\subseteq 2^\Omega$ be an algebra. A content on $S$ is a mapping $\mu:S\to [0,\infty]$ such that $\mu(\emptyset)=0$ and $\mu(A\cup B)=\mu(A)+\mu(B)$ for $A,B\in S$ with $A\cap B=\emptyset$.

We say that $\mu$ is a probability content if $\mu(\Omega)=1$.
\end{definition}
We mainly denote probability contents by the symbol $\PP$. Again, we mainly refer to \cite{RR1983} as a standard reference for the theory of contents.\\

The concrete approach that we now take for a formal system for probability contents on algebras is to introduce an additional constant 
\begin{itemize}
\item $\PP$ of type $1(S)$
\end{itemize}
to the language of the system $\mathcal{F}^\omega$. The first defining properties of $\PP$ as a probability content are then easily formalized in the underlying language as
\begin{align*}
& \forall A^S (0 \le_{\RR} \PP(A) \le_{\RR} 1), \tag*{$(\PP)_1$}\\
& \PP(\emptyset) =_{\RR} 0. \tag*{$(\PP)_2$}
\end{align*}
These statements $(\PP)_1$ and $(\PP)_2$ are again purely universal statements and therefore immediately admissible in the context of metatheorems based on the (monotone) functional interpretation.

The last property of $\PP$, i.e.\ additivity, if formalized naively via
\[
\forall A^S, B^S(A\cap B=_S\emptyset\to\PP(A \cup B) =_{\RR} \PP(A) + \PP(B)),
\]
is not purely universal based on the internal definition of $=_S$ and is instead equivalent over $\mathcal{F}^\omega$ (extended with the constant $\PP$) to the following generalized $\Pi_3$-sentence:
\[
\forall A^S, B^S\exists x^\Omega(x\not\in A\cap B\to\PP(A \cup B) =_{\RR} \PP(A) + \PP(B)).
\]
Similar to how the class of so-called type $\Delta$ sentences is treated in e.g.\ \cite{GuK2016}, it is clear that this statement would be admissible in the context of bound extraction theorems based on the monotone Dialectica interpretation if the $x$ could be conceived of as being bounded in a suitable sense relative to $A$ and $B$. Now, as a matter of fact, a crucial perspective for our formal approach to deriving bound extraction theorems for these systems for algebras and probability contents will be that the whole space $\Omega$ can be naturally regarded as uniformly bounded and, in the context of a corresponding suitable extension of the notion of majorizability to $\Omega$ which reflects this perspective via assuming that there is a uniform majorant for all $x^\Omega$, the above axiom actually has a trivial monotone functional interpretation and is thus admissible in the context of the approach to proof mining metatheorems via such a variant of the Dialectica interpretation. This will be discussed in full formal detail later on so that here, we for now are content with just considering quantification over $\Omega$ as ``bounded'' and so as ``proof-theoretically harmless''.\\

However, if we were to admit the above sentence as the sole axiom, we would be tasked with deriving all the other properties of $\PP$ from this axiom, including monotonicity and thus extensionality of the content as a function, which would require many subtle manipulations of various equalities using the quantifier-free extensionality rule. We therefore instead opt for the following axiomatization which eases the formal development of these properties in the resulting system: For one, instead of additivity, we instead add the following generalized additivity law which holds for probability contents on algebras:
\[
\forall A^S, B^S(\PP(A \cup B) =_\mathbb{R} \PP(A) + \PP(B)-\PP(A\cap B)).\tag*{$(\PP)_3$}
\]
This statement is purely universal and thus immediately admissible in the context of our approach to bound extraction theorems as before. The other property of $\PP$ that we then axiomatically add is that of the monotonicity of $\PP$, i.e.
\[
\forall A^S,B^S\left( A\subseteq_S B\rightarrow \PP(A)\leq_\mathbb{R}\PP(B)\right).
\]
Similar to above, this statement is equivalent to the following (generalized) $\Pi_3$-statement
\[
\forall A^S,B^S\exists x^\Omega\left(\PP(A)>_\mathbb{R}\PP(B)\to x\in A\land x\not\in B\right)\tag*{$(\PP)_4$}
\]
which in the context of the previously sketched extended notion of majorizability will later have a trivial monotone functional interpretation and thus be admissible in the context of our approach to proof mining metatheorems.\\

Adding these statements as axioms to the underlying system for algebras of sets, we derive the following system for probability contents on such algebras:

\begin{definition}
We write $\mathcal{F}^\omega[\PP]$ for the system resulting from $\mathcal{F}^\omega$ by adding the above constant $\PP$ together with the axioms $(\PP)_1$ -- $(\PP)_4$. 
\end{definition}

We now begin with some immediate properties of $\PP$ provable in the above system.

\begin{proposition}
\label{prop:properties_of_P}
The following properties of $\PP$ are provable in $\mathcal{F}^\omega[\PP]$:
\begin{enumerate}
\item $\PP$ is extensional w.r.t.\ $=_S$ and $=_\mathbb{R}$, i.e.
\[
\forall A^S,B^S\left( A=_SB\rightarrow \PP(A)=_\mathbb{R}\PP(B)\right).
\]
\item $\PP$ is definite on $\emptyset$, i.e.
\[
\forall A^S\left(\PP(A)>_\mathbb{R}0\rightarrow A\neq_S\emptyset\right).
\]
\item $\PP$ is additive, i.e.
\[
\forall A^S, B^S(A\cap B=_S\emptyset\rightarrow \PP(A \cup B) =_\mathbb{R} \PP(A) + \PP(B)).
\]
\item $\PP$ respects the relative complements of subsets, i.e.
\[
\forall A^S, B^S(B \subseteq_S A \rightarrow \PP(A \cap B^c) =_\mathbb{R} \PP(A) - \PP(B)).
\]
In particular, we also have
\[
\forall A^S (\PP(A^c) =_\mathbb{R} 1 - \PP(A)).
\]
\item $\PP$ satisfies Boole's inequality, i.e.
\[
\forall A^{S(0)}, n^0\,\left(\PP\left(\bigcup^n_{i=0} A(i)\right) \le_\RR \sum_{i=0}^n \PP(A(i))\right).
\]
\end{enumerate}
\end{proposition}
\begin{proof}
\begin{enumerate}
\item Assume $\PP(A)>\PP(B)$. By axiom $(\PP)_4$, there exists an $x$ such that $x\in A$ and $x\not\in B$, i.e.\ $A\neq B$. Similarly we derive $A\neq B$ from $\PP(A)<\PP(B)$. Combined, we get that $A= B$ implies $\PP(A)=\PP(B).$
\item Assume $\PP(A)>0=\PP(\emptyset)$. Then by axiom $(\PP)_4$, we get an $x\in A$, i.e.\ $A\neq \emptyset$.
\item Let $A,B$ be arbitrary with $A\cap B=\emptyset$. By axiom $(\PP)_3$, we have $\PP(A\cup B)=\PP(A)+\PP(B)-\PP(A\cap B)$. As $\PP$ is extensional, we get $\PP(A\cap B)=\PP(\emptyset)=0$ so that the above implies $\PP(A\cup B)=\PP(A)+\PP(B)$ as desired.
\item Let $E := A \cap B$ and $F:= A \cap B^c$. Then $E \cap F = \emptyset$ (by the properties of algebras of sets). Thus, by additivity, $\PP(E \cup F) = \PP(E) + \PP(F)$. We have that $E \cup F = A$ (again by the properties of algebras of sets). Thus, by extensionality of $\PP$, we have $\PP(A) = \PP(A \cap B) + \PP(A \cap B^c)$. Now, $B \subseteq A$ implies $A \cap B = B$ by Proposition \ref{pro:basic_prop_fields} and so the result follows from the extensionality of $\PP$. 
\item This follows via a simple induction from the axiom $(\PP)_3$.
\end{enumerate}
\end{proof}

Contents on algebras enjoy certain continuity properties similar to continuity from above and below  for measures but without the existence of limiting sets, i.e.\ infinite unions, etc. (see e.g.\ \cite{RR1983}) and we now discuss how the system $\mathcal{F}^\omega[\PP]$ recognizes Cauchy-variants of these properties.\\

For that, we introduce the following operation on terms of type $S(0)$ that allows for the implicit quantification over a disjoint countable family of sets: given $A^{S(0)}$, we set $(A\!\uparrow)(0)=A(0)$ and
\[
(A\!\uparrow)(n+1):=A(n+1)\cap \left(\bigcup^n_{i=0} A(i)\right)^c.
\]
This operation thus turns $A$ into a sequence of disjoint sets $A\!\uparrow$ with the same (partial) union(s) and if $A$ was already a disjoint family, then it is left unchanged by the operation.\\

We now begin with a Cauchy-type form of $\sigma$-additivity of $\PP$ as a content. For this, note that for a given $A^{S(0)}$, the sequence of partial sums  
\[
\sum^n_{i=0}\PP((A\!\uparrow)(i))=\PP\left(\bigcup_{i=0}^n(A\!\uparrow)(i)\right)=\PP\left(\bigcup_{i=0}^nA(i)\right)
\]
is a monotone and bounded sequence of real numbers and thus is Cauchy. The following result that (already a weak fragment of) $\WEPAomega$ suffices to prove the Cauchy formulation of the convergence of monotone and bounded sequences is well known:
\begin{lemma}[folklore, see essentially \cite{Koh2008}]
The system $\WEPAomega$ (and actually already a weak fragment thereof) proves that 
\begin{align*}
&\forall a^{1(0)}\big(\forall n^0\left(0\le_\mathbb{R} a(n)\le_\mathbb{R} 1\land a(n) \le_\mathbb{R} a(n+1)\right)\\
&\qquad\to \forall k^0\exists N^0\forall n^0,m^0\ge_0 N\left(\vert a(n)-a(m)\vert <_\mathbb{R} 2^{-k}\right)\big).
\end{align*}
\end{lemma}

So, instantiating the above result with $a(n)=\sum^n_{i=0}\PP((A\!\uparrow)(i))$, we can derive that $\mathcal{F}^\omega[\PP]$ (and actually already a weak fragment thereof) can prove the Cauchy-property of sequences of contents of increasing disjoint unions:

\begin{proposition}\label{pro:meta_countable_convergence}
The system $\mathcal{F}^\omega[\PP]$ (and actually already a weak fragment thereof) proves
\[
\forall A^{S(0)}\forall k^0\exists N^0\forall n^0,m^0\ge_0 N\left( \left\vert \sum_{i=0}^n\PP((A\!\uparrow)(i)) - \sum_{i=0}^m\PP((A\!\uparrow)(i))\right\vert<_\mathbb{R} 2^{-k}\right).
\]
\end{proposition}

From this Proposition \ref{pro:meta_countable_convergence}, we can then immediately derive the following continuity theorems for contents.

\begin{proposition}\label{pro:continuity}
The system $\mathcal{F}^\omega[\PP]$ (and actually already a weak fragment thereof) proves:
\begin{enumerate}
\item $\PP$ is continuous from below, i.e.
\[
\forall A^{S(0)}\left(\forall n^0\left(A(n) \subseteq_S A(n+1)\right) \to\forall k^0\exists N^0\forall n^0,m^0\ge_0 N\left( \left\vert \PP(A(n)) - \PP(A(m))\right\vert<_\mathbb{R} 2^{-k}\right)\right).
\]
\item $\PP$ is continuous from above, i.e.
\[
\forall A^{S(0)}\left(\forall n^0\left(A(n+1) \subseteq_S A(n)\right) \to\forall k^0\exists N^0\forall n^0,m^0\ge_0 N\left( \left\vert \PP(A(n)) - \PP(A(m))\right\vert<_\mathbb{R} 2^{-k}\right)\right).
\]
\end{enumerate}  
\end{proposition}
\begin{proof}
\begin{enumerate}
\item Note that $A(n) \subseteq A(n+1)$ for all $n$ implies that 
\[
(A\!\uparrow)(n+1) = A(n+1)\cap A(n)^c
\]
for any $n$. Thus by Proposition \ref{prop:properties_of_P}, (4), we have 
\[
\PP((A\!\uparrow)(n+1)) = \PP(A(n+1))- \PP(A(n))
\]
for all $n$. Thus we have
\[
\sum_{i=0}^n\PP((A\!\uparrow)(i))=\PP(A(0))+\sum_{i=0}^{n-1}\left(\PP(A(i+1))-\PP(A(i))\right)=\PP(A(n))
\]
for any $n$. The result now follows from Proposition \ref{pro:meta_countable_convergence}. 
\item Observe that $A(n+1) \subseteq A(n)$ for any $n$ implies that $A(n)^c \subseteq A(n+1)^c$ for all $n$. Thus, by $(1)$, we have 
\[
\forall k^0\exists N^0\forall n^0,m^0\ge_0 N\left( \left\vert \PP(A(n)^c) - \PP(A(m)^c)\right\vert<_\mathbb{R} 2^{-k}\right).
\]
By Proposition \ref{prop:properties_of_P}, (4), we get $\PP(A(n)^c)=1-\PP(A)$ and from this the result follows.
\end{enumerate}
\end{proof}

\section{Systems for $\sigma$-algebras and probability measures}

If we now further require the closure of the underlying algebra of sets under countable unions, we arrive at the notion of a $\sigma$-algebra which forms the algebraic basis for probability measures.

\begin{definition}[$\sigma$-algebra]
Let $\Omega$ be a set and $S\subseteq 2^\Omega$ be an algebra. Then $S$ is called a $\sigma$-algebra if for any $(A_n)\subseteq S$, it also holds that $\bigcup_{n=0}^\infty A_n\in S$.
\end{definition}

The requirement that a content on a $\sigma$-algebra is also well-behaved w.r.t.\ these countable unions then leads to the notion of a measure on such an algebra.

\begin{definition}[Measure]
Let $\Omega$ be a set and $S\subseteq 2^\Omega$ be a $\sigma$-algebra. A measure on $S$ is a content $\mu:S\to [0,\infty]$ that is also $\sigma$-additive, i.e.
\[
\mu\left(\bigcup_{n=0}^\infty A_n\right)=\sum_{n=0}^\infty\mu(A_n)
\]
for any $(A_n)\subseteq S$ with $A_i\cap A_j=\emptyset$ for $i\neq j$.

The map $\mu$ is called a probability measure if $\mu(\Omega)=1$. In that case, $(\Omega,S,\mu)$ is called a probability space.
\end{definition}

In this section, we will now discuss how the previous system for algebras $\mathcal{F}^\omega$ and its extension $\mathcal{F}^\omega[\PP]$ for treating probability contents can be augmented by a certain intensional treatment of countably infinite unions to provide an apt and tame formal basis for these notions.

\subsection{Treating infinite unions tamely}

Concretely, to treat countably infinite unions over algebras of sets tamely, we now extend the previous system $\mathcal{F}^\omega$ with the following further constant
\begin{itemize}
\item $\bigcup$ of type $S(S(0))$,
\end{itemize}
providing a term of type $S$ for the resulting union of the sequence of sets coded by the input of type $S(0)$. So, in the context of suitable axioms specifying that $\bigcup A$ for a given $A^{S(0)}$ represents the union of all $A(n)$, we can formally induce that the algebra $S$ is closed under these countable unions. The immediate axioms specifying the property that $\bigcup A$ is the corresponding union are
\[
\forall A^{S(0)}\forall n^0\left( A(n)\subseteq_S \bigcup A\right) \tag*{$(\bigcup)_1$}
\]
as well as
\[
\forall A^{S(0)}\forall B^S\left( \forall n^0 \left( A(n)\subseteq_S B\right)\to \bigcup A\subseteq_S B\right),
\]
specifying that $\bigcup$ is the join of the elements $A(n)$ in $S$ (seen as a lattice). The first statement $(\bigcup)_1$ is immediately admissible in the context of the Dialectica interpretation as it is purely universal. The latter statement is naturally not admissible as is in the context of the usual approach to proof mining metatheorems as it contains a negative universal quantifier of type $0$ that can not be majorized (which, after all, is also why a uniform variant of arithmetical comprehension is necessary to fully define countable unions of sets, see e.g.\ \cite{Kre2015} for further discussions).\\

Since we want to avoid this strong form of comprehension as to not in general distort the strength of the extracted bounds to be able to a priori guarantee the extractability of proof-theoretically tame bounds from proofs, we opt for the next best thing we can do and instead specify the union only intensionally by adding the following rule-variant of the above converse implication
\[
\frac{F_{qf}\to \forall n^0\left( A(n)\subseteq_SB\right)}{F_{qf}\to \bigcup A\subseteq_S B}\tag*{$(\bigcup)_2$}
\]
where $A$ is a term of type $S(0)$, $B$ is a term of type $S$ and $F_{qf}$ is a quantifier-free formula. So: If $A(n)$ is provably bounded above by $B$ w.r.t.\ $\subseteq_S$ under some quantifier-free assumptions, then also $\bigcup A$ is probably bounded above by $B$ w.r.t.\ $\subseteq_S$ under the same assumptions.

\begin{definition}
We write $\mathcal{F}^\omega[\bigcup]$ for the system resulting from $\mathcal{F}^\omega$ by extending it with the constant $\bigcup$ together with the axiom $(\bigcup)_1$ and the rule $(\bigcup)_2$. Similarly, we write $\mathcal{F}^\omega[\bigcup,\PP]$ for the system that arises from $\mathcal{F}^\omega[\PP]$ by adding the same constants, axioms and rules.
\end{definition}

Using this intensional approach to countable unions, we can also immediately provide an intensional treatment of countable intersections. For this, we first define $A^c$ for an $A^{S(0)}$ by setting $A^c(n):= A(n)^c$ for any $n^0$. Using this notation, we then define the countable intersection of a collection of sets represented by an $A^{S(0)}$ via
\[
\bigcap A := \left(\bigcup A^c\right)^c.
\]

We then get that analogs of the axiom $(\bigcup)_1$ and the rule $(\bigcup)_2$, formulated appropriately for the intersection, are provable in our system $\mathcal{F}^\omega[\bigcup]$:

\begin{lemma}\label{lem:intersection}
The following statement is provable in $\mathcal{F}^\omega[\bigcup]$:
\[
\forall A^{S(0)}\forall n^0\left( \bigcap A\subseteq_S A(n)\right).
\]
Further, in $\mathcal{F}^\omega[\bigcup]$ the following rule is derivable:
\[
\frac{F_{qf}\to \forall n^0\left( B\subseteq_S A(n)\right)}{F_{qf}\to B\subseteq_S \bigcap A}.
\]
\end{lemma}
\begin{proof}
For the provability of the first statement, let $x\in\bigcap A$. By definition we have $x\not\in \bigcup A^c$. Then by axiom $(\bigcup)_1$, we get $x\not\in A(n)^c$, i.e.\ $x\in A(n)$ for any $n$.\\

Now, for the rule, suppose that we provably have $F_{qf}\to\forall n\left( B\subseteq A(n)\right)$. Then we also provably have $F_{qf}\to \forall n(A(n)^c\subseteq B^c)$ and using the rule $(\bigcup)_2$, we get $F_{qf}\to \bigcup A^c\subseteq B^c$. Thus also $F_{qf}\to B\subseteq\left(\bigcup A^c\right)^c=\bigcap A$ as desired.
\end{proof}

\subsection{Handling probability measures}

As we have seen in Propositions \ref{pro:meta_countable_convergence} and \ref{pro:continuity}, the system $\mathcal{F}^\omega[\PP]$ already provides Cauchy-variants of the convergence of monotone sequences of events as well as of sums of disjoint events. In the theory of measures on $\sigma$-algebras, the resulting limits of course correspond to the measure of respective infinite unions or intersections. Thus, the natural question of whether and how this can be formally represented in the system $\mathcal{F}^\omega[\bigcup,\PP]$ immediately arises. And while for a disjoint family represented by $A^{S(0)}$, the limit
\[
0\leq \PP\left(\bigcup A\right)-\PP\left(\bigcup_{i=0}^nA(i)\right)\to 0\text{ for }n\to\infty
\]
holds true, one in general does not have a computable rate of convergence for this expression in the sense that even a function $\varphi$ of type $1$ such that
\[
\forall k^0\exists n\leq_0 \varphi(k)\left(\PP\left(\bigcup A\right)-\PP\left(\bigcup_{i=0}^nA(i)\right)\leq_\mathbb{R} 2^{-k}\right)\tag{$\circ$}
\]
is in general not computable (see Remark \ref{rem:noRate} for an example). Nevertheless, we want to point out that while therefore the convergence of the sequence $\sum_{i=0}^n\PP(A(i))$ towards $\PP(\bigcup A)$ can not be provable in any system that allows for the extraction of computable and uniform bounds, the system $\mathcal{F}^\omega[\bigcup,\PP]$ nevertheless provides an intensional version of that convergence in the following sense:
\begin{enumerate}
\item By Proposition \ref{pro:meta_countable_convergence}, the sequence of partial sums $\sum_{i=0}^n\PP(A(i))$ is provably Cauchy.
\item Using the additivity and monotonicity of $\PP$, i.e.\ axioms $(\PP)_3$ and $(\PP)_4$, we get by $\bigcup_{i=0}^nA(i)\subseteq_S \bigcup A$ that
\[
\sum_{i=0}^n\PP(A(i))=_\mathbb{R}\PP\left(\bigcup_{i=0}^n A(i)\right)\leq_\mathbb{R} \PP\left(\bigcup A\right)
\]
holds provably.
\item For any object $B^S$ such that we provably have $\forall n^0\left(A(n)\subseteq_S B\right)$, we get $\bigcup A\subseteq_S B$ using the rule $(\bigcup)_2$ and so we get provably
\[
\PP\left(\bigcup A\right)\leq_\mathbb{R} \PP(B)
\]
in that case by monotonicity of $\PP$.
\end{enumerate}
So the value $\PP(\bigcup A)$ is at least intensionally specified to be the limit of the partial sums $\sum_{i=0}^n\PP(A(i))$ as $\PP(\bigcup A)$ is bounded below by this nondecreasing sequence of partial sums and intensionally bounded above by the probability of any set which provably sits above the given partial unions $\bigcup_{i=0}^nA(i)$.\\

The case that we want to make is now twofold: For one, as mentioned in the introduction, the theory of contents already exhaust large parts of the theory of measures in the sense that often already the properties of contents on algebras suffice to carry out proofs for properties of measures on $\sigma$-algebras (as will also be the case in the applications discussed later). For another, we want to argue that even in situations where one can not do just with finite unions and content, such an intensional specification of countable unions and their measures might suffice for formalizing a given proof and all the while guaranteeing the extractability of tame bounds bounds a priori. If that is not the case, then the result under consideration might be considered to be inherently ``untame'' and a full treatment of the comprehension principle needed to define the respective unions will be necessary. We therefore regard $\mathcal{F}^\omega[\bigcup,\PP]$ as a suitable tame base system for treating probability measures on $\sigma$-algebras.

\begin{remark}\label{rem:noRate}
For an example where $\varphi$ in ($\circ$) is not computable, we take $(r_i)\subseteq (0,1)$ to be a sequence of computable real numbers such that 
\[
a_n=\sum_{i=0}^nr_i\to a \le 1 
\]
without a computable rate of convergence.\footnote{The existence of such a sequence is due to \cite{Spe1949}. For a direct construction, we proceed similar as in \cite{TvD1988}: pick an enumeration $f:\mathbb{N}\to\mathbb{N}$ of the special Halting problem without repetitions. Then defining $r_i=2^{-f(i)-1}$ yields a suitable sequence such that $a_n$ defined as above naturally converges to an element $a\in [0,1]$ as it is monotone and bounded above but the rate of convergence can not be computable as this would allow one to decide the special Halting Problem.} We now define $\Omega=\mathbb{N}\cup\{\infty\}$ as well as $S=2^\Omega$. On the discrete sample space $\Omega$, we then define a probability mass function $p$ via $p(i)=r_{i}$ for $i\in\mathbb{N}$ as well as $p(\infty)=1-a$. This $p$ then as usual induces a probability measure $\PP$ on the $\sigma$-algebra $S$ defined for $A\subseteq \Omega$ via
\[
\PP(A)=\sum_{a\in A}p(a).
\]
Clearly $(\Omega,S,\PP)$ is a probability space where
\[
\PP(\mathbb{N})-\sum_{i=1}^{n+1}\PP(\{i\})=a-\sum_{i=0}^nr_i=a-a_n
\]
can not have a computable rate of convergence to $0$.
\end{remark}

\section{Intensional intervals, inverse mappings and measurable functions}\label{sec:measurable}

In this section, we now extend the machinery of the previous logical systems so that we are able to deal with functions $f:\Omega\to\mathbb{R}$ that are measurable in certain suitable sense relative to algebras. As such, the treatment given here will be instrumental for both our approach to integrable functions in Section \ref{sec:integral} as well as to proof-theoretic transfer principles for implications between modes of convergence in Section \ref{sec:transfer}. For this, we now first recall the essential definitions and basic results.

\begin{definition}[Borel $\sigma$-algebra]
Let $Y$ be a topological space. The Borel $\sigma$-algebra $\mathcal{B}(Y)$ on $Y$ is the smallest $\sigma$-algebra on $Y$ that contains all open subsets of $Y$.
\end{definition}

We refer to \cite{Hal1974} as a standard reference for Borel $\sigma$-algebras in particular and measure theory in general (in particular regarding to the well-definedness of the above definition for which one needs to see that the intersection of any family of ($\sigma$-)algebras is again a ($\sigma$-)algebra).\\

Crucial for us will be the notion of a generating set of a ($\sigma$-)algebra.

\begin{definition}[Generators of a ($\sigma$-)algebra]
Let $\Omega$ be a set and $S\subseteq 2^\Omega$ be a ($\sigma$-)algebra. A generating set for $S$ is a set $S_0\subseteq 2^\Omega$ such that $S$ is the smallest ($\sigma$-)algebra containing $S_0$.
\end{definition}

In that terminology, the Borel $\sigma$-algebra $\mathcal{B}(Y)$ is the $\sigma$-algebra generated by the open subsets of the underlying topological space.\\

For the special case of the real numbers as a topological space with the usual topology induced by the metric distance, we in particular get the following canonical generator besides the open subsets of $\mathbb{R}$.

\begin{lemma}[folklore, see e.g.\ \cite{Hal1974}]
The Borel $\sigma$-algebra $\mathcal{B}(Y)$ on $\mathbb{R}$ is generated by the collection of all closed intervals $\{[a,b]\mid a,b\in\mathbb{R}\}$.
\end{lemma}

One then arrives at the notion of a measurable function which we for simplicity only formulate for real valued functions here.

\begin{definition}[Borel-measurable function]
Let $(\Omega,S,\mu)$ be a content space. A function $f:\Omega\to\mathbb{R}$ is called Borel-measurable if
\[
\{x\in \Omega\mid f(x)\in B\}=:f^{-1}(B)\in S
\]
for all $B\in \mathcal{B}(\mathbb{R})$.
\end{definition}

As we will crucially use later on, Borel-measurability is simply characterized by a similar condition on a generating set.

\begin{proposition}[folklore, see e.g.\ \cite{Hal1974}]
Let $(\Omega,S,\mu)$ be a measure space. A function $f:\Omega\to\mathbb{R}$ is Borel-measurable if, and only if, $f^{-1}([a,b])\in S$ for all $a,b\in\mathbb{R}$.
\end{proposition}

If the underlying algebra $S$ is not a $\sigma$-algebra and if $\mu$ is only a content, then the requirement that the preimages of the above collection of intervals are included is still statable but might result in something weaker than full Borel-measurability. We call a function $f:\Omega\to\mathbb{R}$ where the preimages of all intervals $[a,b]$ for $a,b\in\mathbb{R}$ are included in a corresponding algebra $S\subseteq 2^\Omega$ to be \emph{weakly Borel-measurable}. It will in particular be this notion of weakly Borel-measurable functions that we will rely on later in the context of our approach towards Lebesgue integrals for probability contents. It should be noted that by requiring the inclusion $f^{-1}([a,b])\in S$ for two real numbers $a,b\in\mathbb{R}$ and an algebra $S$, we in particular also obtain that $f^{-1}([a,b))=f^{-1}([a,b])\cap \left(f^{-1}([b,b])\right)^c\in S$ as $S$ is an algebra.\footnote{Note that in the context of algebras, this is a particular benefit from working with closed intervals in the above notion of weak Borel-measurability. If, e.g., one instead would work with half-open intervals $[a,b)$, then defining the closed intervals requires the use of a countably infinite intersection via $[a,b]=\bigcap_{k\in\mathbb{N}}[a,b+2^{-k})$ which can only be sustained in $\sigma$-algebras.}\\

To formally deal with the notion of (weak) Borel-measurability, we thus need an access to the collection of the closed intervals $[a,b]$ for $a,b\in\mathbb{R}$ generating the Borel-algebra. For this, we will introduce an intensional approach to real intervals in the next subsection to provide formal means of operating on these generators. These intensional variants of real intervals can then be processed by a general type of inverse map using which we will be able to state the measurability of a function formally.

\subsection{Intensional Intervals}

Concretely, we now provide a quantifier-free (and thus in a way intensional) account of the closed intervals $[a,b]$ (and thus also of the half-open intervals $[a,b)$ as discussed above) by introducing a further constant to the language:
\begin{itemize}
\item $[\cdot,\cdot]$ of type $0(1)(1)(1)$.
\end{itemize}
Given two inputs $a^1, b^1$, this function shall return a characteristic function for an intensional representation of the corresponding interval. For this, we add the following axioms:
\begin{align*}
&\forall a^1,b^1,r^1\left([a,b](r)\leq_0 1\right), \tag*{$([\cdot,\cdot])_1$}\\
&\forall a^1,b^1,r^1\left(r\in [a,b]\to a\leq_\mathbb{R} r\leq_\mathbb{R}b\right), \tag*{$([\cdot,\cdot])_2$}\\
&\forall a^1,b^1,r^1\left(a<_\mathbb{R}r<_\mathbb{R} b\to r\in [a,b]\right), \tag*{$([\cdot,\cdot])_3$}\\
&\forall a^1,b^1\left(a,b\in [a,b]\right). \tag*{$([\cdot,\cdot])_4$}
\end{align*}
Here, we wrote $[a,b]$ for $[\cdot,\cdot]ab$ as well as $r\in [a,b]$ for $[a,b](r)=_00$. Note that this is an intensional representation of the set as we have $a,b\in [a,b]$ but we can not conclude from $r=a$ or $r=b$ that $r\in [a,b]$.\\

We also introduce the following notation used later: if we are given a system $\mathcal{C}^\omega$, we write $\mathcal{C}^\omega[\mathrm{Int}]$ to denote the extension of $\mathcal{C}^\omega$ by the above constant $[\cdot,\cdot]$ and the axioms $([\cdot,\cdot])_1$ -- $([\cdot,\cdot])_4$ for treating the closed intervals.\\

In similarity to the discussion above, these closed intervals then also provide a quantifier-free access to the half-open intervals in the following way: We define $[\cdot,\cdot)$ of type $0(1)(1)(1)$ via
\[
[\cdot,\cdot):=\lambda a^1,b^1,r^1.\max\left\{[a,b](r),1-[b,b](r)\right\}.
\]
Also for $[\cdot,\cdot)$, we write $[a,b)$ for $[\cdot,\cdot)ab$ as well as $r\in [a,b)$ for $[a,b)(r)=_00$.\\

In the context of that definition, we in particular obtain relatively immediately that $[\cdot,\cdot)$ defined as such satisfies properties that intensionally specify the half-open intervals similar as to how we have specified closed intervals above with the axioms $([\cdot,\cdot])_1$ -- $([\cdot,\cdot])_4$, i.e.\ we for one have $a\in [a,b)$ but from $r=a$, we can not infer $r\in [a,b)$ and we have $b\not\in [a,b)$ but from $r=b$, we can not infer $r\not\in [a,b)$. This is collected in the following lemma:

\begin{lemma}
The system $\mathcal{A}^\omega[\mathrm{Int}]$ proves the following properties of $[\cdot,\cdot)$:
\begin{enumerate}
\item $\forall a^1,b^1,r^1\left([a,b)(r)\leq_0 1\right)$,
\item $\forall a^1,b^1,r^1\left(r\in [a,b)\to a\leq_\mathbb{R} r\leq_\mathbb{R}b\right)$,
\item $\forall a^1,b^1,r^1\left(a<_\mathbb{R}r<_\mathbb{R} b\to r\in [a,b)\right)$,
\item $\forall a^1,b^1\left(a<_\mathbb{R}b\to a\in [a,b)\land b\not\in [a,b)\right)$.
\end{enumerate}
\end{lemma}

We omit the proof as it is rather immediate.

\subsection{The inverse map}\label{sec:invMap}

We now provide a treatment of the inverse map for a given function $f$ of type $1(\Omega)$. For this, we actually introduce a uniform operator into the language via a constant
\begin{itemize}
\item $(\cdot)^{-1}$ of type $0(\Omega)(0(1))(1(\Omega))$
\end{itemize}
that provides an inverse map for any given function $f^{1(\Omega)}$ in the sense that, writing $f^{-1}$ for $(\cdot)^{-1}f$, the functional $f^{-1}$ receives a subset of the reals coded via a characteristic function $A^{0(1)}$ and maps this into a characteristic function $f^{-1}A$ of type $0(\Omega)$ coding a subset of the underlying space $\Omega$.\\

This type of map is then governed by the following two axioms:
\begin{align*}
&\forall f^{1(\Omega)}\forall A^{0(1)}\forall x^\Omega\left(f^{-1}Ax\leq_0 1\right), \tag*{$(\textrm{Inv})_1$}\\
&\forall f^{1(\Omega)}\forall A^{0(1)}\forall x^\Omega\left( x\in f^{-1}A\leftrightarrow f(x)\in A\right). \tag*{$(\textrm{Inv})_2$}
\end{align*}
Here, we wrote $f(x)\in A$ for $Af(x)=_00$ and $x\in f^{-1}A$ for $f^{-1}Ax=_00$.\\

Also for this type of extension, we introduce the following generic notation: given a system $\mathcal{C}^\omega$, we write $\mathcal{C}^\omega[\mathrm{Inv}]$ to denote the extensions of $\mathcal{C}^\omega$ by the constant $(\cdot)^{-1}$ and the above axioms for treating the inverse map.

\subsection{Measurability of functions}

In the context of the intensional representations for closed intervals generating the Borel $\sigma$-algebra on the reals as well as using the general inverse map introduced before, we are now in the position of formulating the (weak) Borel-measurability of a function $f^{1(\Omega)}$ formally in the underlying language by
\[
\forall a^1,b^1\exists A^S\forall x^\Omega\left( x\in f^{-1}([a,b])\leftrightarrow x\in A\right).
\]
The inner matrix (based on the fact that stating element relations via $\in$ is quantifier-free and as we use the abbreviation $x\in f^{-1}([a,b])$ for $f^{-1}[a,b]x=_00$ as introduced in Section \ref{sec:invMap}) is  quantifier-free and thus the above sentence is a generalized $\Pi_3$-statement. Similar as with the monotonicity statement of the content $\PP$, this statement would be admissible in the context of the monotone Dialectica interpretation if the quantification over elements of type $S$ could be conceived of as being bounded in a suitable sense. Similar as we argued with the monotonicity of $\PP$, a crucial perspective for our formal approach to bound extraction theorems later on will be that, besides the whole space $\Omega$, we will also be able to regard the space $S$ as uniformly bounded, formally encapsulated via a corresponding suitable extension of the notion of majorizability to $S$ used later. In that context, such a sentence sentence has a trivial monotone functional interpretation and we will use this later to formulate admissible axioms stating that certain classes of functions are indeed measurable.

\section{Treating integration over probability contents}\label{sec:integral}

We now want to extend the previous system $\mathcal{F}^\omega[\PP]$ for probability contents on algebras so that we can treat a certain class of integrable functions $f:\Omega\to\mathbb{R}$. With that, we thus in particular provide a firm base for random variables and their moments as used in various applications which will be illustrated further by Section \ref{sec:appliations}.\\

For the usual approach to the integral over contents, which mimics that of the Lebesgue integral, we mainly follow the exposition given in \cite{RR1983} (where the corresponding notion is introduced under the name of the ``D-integral'') which we detail here to some degree to provide the necessary mathematical basics for the axiomatizations chosen later. Concretely, let $\Omega$ be a set and $S$ an algebra on it and let $\mu$ be a finite content on $S$ (i.e.\ $\mu(\Omega)<+\infty$). One then first arrives at a notion of simple function that is completely analogous to how it is usually defined in the context of measure theory, i.e.\ a \emph{simple function} is a function $f:\Omega\to\mathbb{R}$ of the form
\[
f(x)=\sum_{i=0}^n b_i\chi_{A_i}
\]
for given sets $A_i\in S$ and values $b_i\in\mathbb{R}$.\footnote{In some cases, as e.g.\ also in \cite{RR1983}, one requires the sets $A_i$ to be mutually disjoint and to cover the whole space $\Omega$ but we do not include these requirements here for simplicity.} For such a function $f$, the integral over $\mu$ is simply defined as
\[
\int{f}\,\mathrm{d}\mu=\sum_{i=0}^n b_i\mu(A_i),
\]
also in similarity to usual Lebesgue integrals over measures. A general function $f:\Omega\to\mathbb{R}$ is now declared \emph{integrable} if there is a sequence of simple functions $f_n$ such that (1) the $f_n$ converge to $f$ in a suitable sense\footnote{The corresponding notion of convergence is dubbed \emph{hazy convergence} in \cite{RR1983} and relies on the use of a corresponding outer content (which is similar to an outer measure) constructed from $\mu$. Here, we will however not rely on any precise details regarding this notion.} and (2) the sequence satisfies
\[
\lim_{n,m\to\infty}\int{\vert f_n-f_m\vert}\,\mathrm{d}\mu=0.
\]
Such a sequence is called a \emph{determining sequence} for $f$ and using such a determining sequence, one then defines the integral of $f$ via
\[
\int{f}\,\mathrm{d}\mu=\lim_{n\to\infty}\int{f_n}\,\mathrm{d}\mu.
\]
Crucially, this limit is well-defined as the following general result shows:
\begin{lemma}[Lemma 4.4.12 in \cite{RR1983}]
Let $f$ be an integrable function and let $(f_n)$ be a determining sequence for $f$. Then $f_n-f$ is integrable and 
\[
\lim_{n\to\infty}\int{\vert f_n-f\vert}\,\mathrm{d}\mu=0.
\]
\end{lemma}

This notion of a (D-)integral defined for contents then shares many of the familiar properties of Lebesgue integrals defined for measures. One particularly useful property is that every integrable function $f$ is measurable in the following extended sense, called $T_2$-measurable in \cite{RR1983}: for any $\varepsilon>0$, there is a partition $A_0,\dots, A_n\in S$ of $\Omega$ such that $\mu(A_0)<\varepsilon$ and $\vert f(x)-f(y)\vert<\varepsilon$ for any $x,y\in A_i$ and any $i$. In particular, for any function $f$ that is measurable in this sense, one gets that $f$ is integrable if, and only if, $\vert f\vert$ is integrable (see e.g.\ Corollary 4.4.19 in \cite{RR1983}).\\

Now, a major part of the theory of integrals over contents (like e.g.\ a nice correspondence between the so-called D- and S-integrals, the latter being similar in spirit to a Riemann-Stieltjes integral, and the fact that the notions of $T_2$-measurability and integrability coincide as shown by Theorem 4.5.7 in \cite{RR1983}, among many others) depends on the assumption that the integrated functions are bounded. In that way, we will similarly require that all integrated function are bounded.

In fact, using the proof-theoretic perspective of the approach taken here to the proof mining metatheorems for contents on algebras, we find that this assumption of the boundedness of functions that is often imposed in the context of integration on contents is also suggested as a necessity in our formal approach by the notion of majorizability employed later. Concretely, as discussed before, to develop a proof-theoretically tame theory of algebras and contents, we have regarded $\Omega$ and $S$ as uniformly bounded in the sense that we later regard all elements of these spaces as uniformly majorized by the content of the full space $\mu(\Omega)$, i.e.\ by $1$ in the context of a probability content, which we denote in writing by $1\gtrsim_\Omega x$ and $1\gtrsim_S A$ for $x\in \Omega$ and $A\in S$. While this will be discussed comprehensively and in full formal detail later, we here look already at what this definition entails for majorizable functions $f$ of type $1(\Omega)$: a function $f^*$ of type $0(0)(0)$ is a majorant for $f$, written $f^*\gtrsim_{1(\Omega)} f$, if
\[
f^*(m)(k)\geq f^*(n)(j)\geq f(x)(i)\text{ whenever }m\geq n\gtrsim_\Omega x\text{ and }k\geq j\geq i.
\]
Therefore, as $1\gtrsim_\Omega x$ for any $x$, we in particular derive that 
\[
f(x)(k)\leq f^*(1)(k)
\]
for any $k$ and $x$ and thus, the real number represented by $f(x)$ is uniformly bounded by $f^*(1)(0)+1$ for any $x$. Thus, any majorizable function of type $1(\Omega)$ is bounded and so boundedness of integrable functions is suggested as a necessary assumption by the chosen proof-theoretic methodology.\\

Lastly, to provide a proof-theoretically tame approach to integration, we will actually require that the integrated functions are not only $T_2$-measurable in the sense discussed above but that they even are weakly Borel-measurable in the sense discussed before (i.e.\ that the preimages of closed intervals are included in the underlying algebra). Clearly, any bounded and weakly Borel-measurable function is also $T_2$-measurable in the previous sense and thus integrable as discussed before and while this class is slightly restricted compared to that of all integrable functions, we find that is has two central advantages: for one, as mentioned before, it allows for a smooth and proof-theoretically tame approach to the integral for this class of functions that is amenable to bound extraction theorems and, for another, in virtually all previous ad hoc proof mining applications to measure and probability theory involving integrals, already the stronger (or, in the context of $\sigma$-algebras, equivalent) property of being Borel-measurable is assumed so that a treatment of this class still allows for our metatheorems established later to provide a proof-theoretic explanation of the respective extractions.\\

Now, the formal approach to the integral over a probability content is to abstractly encode a suitable subspace of bounded and weakly Borel-measurable functions closed under linear combinations, multiplications with characteristic functions and absolute values\footnote{In the context of a probability measure over a $\sigma$-algebra, such a space could for example be the space of all bounded and Borel-measurable functions.} intensionally via the use of a characteristic function and then to introduce the integral for all functions from this space as well as the relevant closure properties using further constants and axioms. For this, we now initially introduce two further constants
\begin{itemize}
\item $I$ of type $0(1(\Omega))$,
\item $\norm{\cdot}_\infty$ of type $1(1(\Omega))$,
\end{itemize}
into the language of $\mathcal{F}^\omega[\PP,\mathrm{Int},\mathrm{Inv}]$. The first of these is the previously mention characteristic function providing an intensional account of a space closed under linear combinations, multiplications with characteristic functions and absolute values as well as containing only bounded and weakly Borel-measurable functions and the latter is used to formally introduce the supremum norm on these functionals. As initial axioms, we now therefore stipulate the following:
\begin{align*}
&\forall f^{1(\Omega)}\left( If\leq_01\right), \tag*{$(I)_1$}\\
&\forall A^S\left( (\lambda x.(x\in A))\in I\right), \tag*{$(I)_2$}\\
&\forall f^{1(\Omega)},a^1,b^1\exists A^S\forall x^\Omega\left(f\in I\to \left(x\in f^{-1}([a,b])\leftrightarrow x\in A\right)\right),\tag*{$(I)_3$}\\
&\forall f^{1(\Omega)}, x^\Omega\left( f\in I\to \left( \vert f(x)\vert\leq_\mathbb{R} \norm{f}_\infty\right)\right),\tag*{$(I)_4$}\\
&\forall f^{1(\Omega)},k^0\exists x^\Omega\left( f\in I\to \left( \norm{f}_\infty-2^{-k}\leq_\mathbb{R} \vert f(x)\vert\right)\right),\tag*{$(I)_5$}\\
&\forall f^{1(\Omega)},g^{1(\Omega)},\alpha^1,\beta^1\left(f,g\in I\to (\lambda x.(\alpha f(x)+\beta g(x)))\in I\right), \tag*{$(I)_6$}\\
&\forall f^{1(\Omega)}\left(f\in I\to \lambda x.\vert f(x)\vert\in I\right),\tag*{$(I)_7$}\\
&\forall f^{1(\Omega)},A^S\left(f\in I\to \lambda x.\left(f(x)(x\in A)\right)\in I\right).\tag*{$(I)_8$}
\end{align*}

The axioms $(I)_3$ and $(I)_5$ will again be admissible later because of the extended notion of majorizability on $\Omega$ and $S$. Note also that axioms $(I)_4$ and $(I)_5$ together specify $\norm{f}_\infty$ as the least upper bound on $\vert f(x)\vert$.\footnote{In fact, these two axioms can be seen as an instantiation of the general approach to tame suprema over bounded sets developed in \cite{Pis2023}.}\\

In the following, we will write $\chi_A$ for the function $\lambda x.(x\in A^c)$.\footnote{As customary in the context of integration, we want characteristic functions to take the value $1$ if the element is included in the set. As we previously have used $0$ for this, we used $A^c$ in the above definition.} Also, we in the following just briefly write $\alpha f+ \beta g$ for $\lambda x.(\alpha f(x)+\beta g(x))$ as well as $\vert f\vert$ for $\lambda x.(\vert f(x)\vert)$ and $f\chi_A$ for $\lambda x.\left(f(x)\chi_A(x)\right)$.\\

As the operations $\max$ and $\min$ can be defined on functions using the absolute value via
\[
\max\{f,g\}=(f+g+\vert f-g\vert)/2\text{ and }\min\{f,g\}=-\max\{-f,-g\},
\]
we immediately get that axiom $(I)_7$ implies the closure of $I$ under these operations and thus we have effectively axiomatized that $I$ in particular is a Riesz space of bounded functions with the respective operations and thus our approach is similar as to how abstract integration spaces are approached in the context of Daniell integrals \cite{Dan1918} where a similar collection of functions is presumed.\\

To deal with the integral, we add a further constant
\begin{itemize}
\item $\INT{\cdot}$ of type $1(1(\Omega))$.
\end{itemize}
The first two axioms for the integral are now that it behaves as expected on characteristic functions and that the integral is a linear function:
\begin{align*}
&\forall A^S\left( \INT{\chi_A}=_\mathbb{R}\PP(A)\right), \tag*{$(\int)_1$}\\
&\forall f^{1(\Omega)}, g^{1(\Omega)},\alpha^1,\beta^1\left( f,g\in I\to \INT{(\alpha f+\beta g)}=_\mathbb{R} \alpha\INT{f}+\beta\INT{g}\right).\tag*{$(\int)_2$}
\end{align*}
Using these two axioms, we immediately get that the integral behaves as expected on simple functions.\\

The major other statement that we need to axiomatize is that any function in $I$ is actually integrable in the sense that their integrals are well-defined and arise as the limit of a sequence of integrals of simple functions. To formulate this property, we now fully utilize the assumption that our functions are bounded and weakly Borel-measurable to derive the following general result that inspires the subsequent axiom: 

\begin{lemma}[essentially folklore]
Let $\Omega$ be a set, $S$ an algebra on $\Omega$ and $\mu$ a probability content on $S$. Let $f$ be a weakly Borel-measurable function and assume that $\vert f\vert$ is bounded by $b\in\mathbb{N}^*$, i.e.\ $\vert f(x)\vert \leq b$ for all $x\in \Omega$. For a given $k$, define
\[
I_{k,i}=\left[-b+\frac{i}{2^k},-b+\frac{i+1}{2^k}\right) \text{ for }i=0,1,\dots,2b2^k-2\text{ and }I_{k,2b2^k-1}=\left[\frac{b2^k-1}{2^k},b\right].
\]
Then:
\[
\forall k\in\mathbb{N}\left(\INT{\left\vert f-\sum_{i=0}^{2b2^k-1}\left(-b+\frac{i}{2^{k}}\right)\chi_{f^{-1}(I_{k,i})}\right\vert}\leq 2^{-k}\right).
\]
\end{lemma}
\begin{proof}
Let $k\in\mathbb{N}$. As all $I_{k,i}$ are disjoint and cover $[-b,b]$ and since $\vert f\vert$ is bounded by $b$, their preimages under $f$ are disjoint and cover $\Omega$. Thus
\[
1=\PP(\Omega)=\PP\left(\bigcup_{i=0}^{2b2^k-1}f^{-1}(I_{k_i})\right)=\sum_{i=0}^{2b2^k-1}\PP(f^{-1}(I_{k_i})).
\]
and for any $k\in\mathbb{N}$:
\[
f(x)=\sum_{i=0}^{2b2^k-1} f(x)\chi_{f^{-1}(I_{k,i})}(x).
\]
Further, for $x\in f^{-1}(I_{k,i})$, it clearly holds that 
\[
\left\vert f(x)-\left(-b+\frac{i}{2^{k}}\right)\right\vert\leq \frac{1}{2^k}.
\]
As $k$ was arbitrary, the function $f$ is integrable by Theorem 4.5.7 in \cite{RR1983} (use e.g.\ the equivalence between (viii) and (v)). Thus, using the monotonicity and linearity of the integral on contents (see e.g.\ Theorem 4.4.13 in \cite{RR1983}), we have 
\begin{align*}
\INT{\left\vert f-\sum_{i=0}^{2b2^k-1}\left(-b+\frac{i}{2^{k}}\right)\chi_{f^{-1}(I_{k,i})}\right\vert}&=\INT{\left\vert\sum_{i=0}^{2b2^k-1}\left(f+b-\frac{i}{2^{k}}\right)\chi_{f^{-1}(I_{k,i})}\right\vert}\\
&\leq \sum_{i=0}^{2b2^k-1}\INT{\left\vert f-\left(-b+\frac{i}{2^{k}}\right)\right\vert\chi_{f^{-1}(I_{k,i})}}\\
&\leq \sum_{i=0}^{2b2^k-1}\INT{\frac{1}{2^k}\chi_{f^{-1}(I_{k,i})}}\\
&\leq \sum_{i=0}^{2b2^k-1}\frac{1}{2^k}\PP(f^{-1}(I_{k,i}))\\
&\leq \frac{1}{2^k}\sum_{i=0}^{2b2^k-1}\PP(f^{-1}(I_{k,i}))=\frac{1}{2^k}.
\end{align*}
\end{proof}

To axiomatize the integrability of a function $f\in I$, it thus suffices to state the conclusion of the above lemma and although we could formalize this directly by employing the general inverse mapping and intensional intervals, we can actually avoid this machinery here at the mild expense of quantifying over the sequence of sets used in the simple functions instead of explicitly specifying them. Concretely, we consider the following third axiom\footnote{Note that in the context of this axiom, we can introduce the sum expression by the recursor constants of the underlying system $\mathcal{A}^\omega$.}
\[
\forall f^{1(\Omega)}\forall k^0\exists A^{S(0)}\left(f\in I\to \INT{\left\vert f-\sum_{i=0}^{2^{k+1}b_f-1}\left(-b_f+\frac{i}{2^{k}}\right)\chi_{A(i)}\right\vert}\leq_\mathbb{R} 2^{-k}\right)\tag*{$(\int)_3$}
\]
where we wrote $b_f:=[\norm{f}_\infty](0)+1$ and have used that, since $\norm{f}_\infty\geq_\mathbb{R} \vert f(x)\vert$ for all $x$, it holds that the natural number $b_f$ similarly bounds $\vert f\vert$. Again, by the later considerations on majorizability whereby also $A$ of type $S(0)$ can be regarded as uniformly bounded, this axiom will later be admissible in the context of our approach to proof mining metatheorems via the monotone functional interpretation.\\

Lastly, to also devise a practical system, it will be convenient to also axiomatically include (instead of discussing how it might be provable in the system) that the integral of a positive function $f\in I$ is positive. Naively, this statement can be written as
\[
\forall f^{1(\Omega)}\left( f\in I\land \forall x^\Omega\left( f(x)\geq_\mathbb{R} 0\right)\to \INT{f}\geq_\mathbb{R} 0\right)
\]
which is not a priori admissible in the context of our approach to bound extraction theorems due to the hidden negative universal type $0$ quantifier in $\geq_{\mathbb{R}}$. However, if we rewrite the above statement in the prenexed form 
\[
\forall f^{1(\Omega)}\forall k^0\exists x^\Omega\exists j^0\left( f\in I\land \INT{f}<_\mathbb{R} -2^{-k}\to \left( f(x)\leq_\mathbb{R} -2^{-j}\right)\right),
\]
we can witness the quantifier over $j$ simply by $j=k$ which can be immediately seen using only very basic properties of the integral (see e.g.\ Theorem 4.4.13 in \cite{RR1983}) so that we arrive at the axiom
\[
\forall f^{1(\Omega)}\forall k^0\exists x^\Omega\left( f\in I\land \INT{f}<_\mathbb{R} -2^{-k}\to\left( f(x)\leq_\mathbb{R} -2^{-k}\right)\right)\tag*{$(\int)_4$}
\]
which, in the context of our perspective that we regard quantification over $\Omega$ as bounded, will later be admissible in the context of the monotone functional interpretation.\\

We want to note that the axioms $(\int)_2$ -- $(\int)_4$ roughly correspond to the four central properties of an abstract ``I-integral'' in the context of Daniell's approach to integration \cite{Dan1918} and so, in some sense, our approach to the integral here can be regarded as a sort of effectivized implementation of the Daniell integral.\\

With all of these constants and axioms, we then arrive at the following system for integrals:

\begin{definition}
We write $\mathcal{F}^\omega[\PP,\mathrm{Integral}]$ for the system resulting from $\mathcal{F}^\omega[\PP,\mathrm{Int},\mathrm{Inv}]$ by adding the above constants $I,\norm{\cdot}_\infty,\INT{\cdot}$ together with the axioms $(I)_1$ -- $(I)_7$ as well as $(\int)_1$ -- $(\int)_4$. Further, we write $\mathcal{F}^\omega[\bigcup,\PP,\mathrm{Integral}]$ for the system $\mathcal{F}^\omega[\PP,\mathrm{Integral}]$ extended with the constant $\bigcup$ and the axiom $(\bigcup)_1$ as well as the rule $(\bigcup)_2$.
\end{definition}

We end this section with some immediate properties of the integral over contents that are provable in our system. A more intricate use of the integrals will then be made in Section \ref{sec:appliations} where they feature crucially in the formal explanation of a previous proof-mining application (and in particular highlight the usability of the above axiomatic approach to the integral in regard to the proof mining practice). As is common in proof mining however, this approach to the integral enjoys a large degree of flexibility in the sense that it can, of course, be immediately augmented by further constants and axioms specifying certain properties of the integral that might be crucial in a certain application.

\begin{lemma}\label{lem:intProp}
The system $\mathcal{F}^\omega[\PP,\mathrm{Integral}]$ proves:
\begin{enumerate}
\item $\INT{\cdot}$ is monotone w.r.t.\ pointwise inequality, i.e.
\[
\forall f^{1(\Omega)}, g^{1(\Omega)}\left( f,g\in I\land \forall x^\Omega\left( f(x)\leq_\mathbb{R} g(x)\right) \to \INT{f}\leq_\mathbb{R}\INT{g}\right).
\]
\item $\INT{\cdot}$ is extensional w.r.t.\ pointwise equality, i.e.
\[
\forall f^{1(\Omega)},g^{1(\Omega)}\left( f,g\in I\land \forall x^\Omega\left( f(x)=_\mathbb{R}g(x)\right)\to \left\vert\INT{f}-\INT{g}\right\vert=_\mathbb{R}0\right).
\]
\item $\INT{\cdot}$ is monotone w.r.t.\ inequality almost everywhere, i.e.
\begin{gather*}
\forall f^{1(\Omega)},g^{1(\Omega)}\bigg( f,g\in I\land \exists A^S\left( \PP(A)=0\land \forall x^\Omega\left( x\in A^c\to f(x)\leq_\mathbb{R}g(x)\right)\right)\\
\to \INT{f}\leq_\mathbb{R}\INT{g}\bigg).
\end{gather*}
\item $\INT{\cdot}$ is extensional w.r.t.\ equality almost everywhere, i.e.
\begin{gather*}
\forall f^{1(\Omega)},g^{1(\Omega)}\bigg( f,g\in I\land \exists A^S\left( \PP(A)=0\land \forall x^\Omega\left( x\in A^c\to f(x)=_\mathbb{R}g(x)\right)\right)\\
\to \left\vert\INT{f}-\INT{g}\right\vert=_\mathbb{R}0\bigg).
\end{gather*}
\item $\INT{\cdot}$ behaves well with absolute values, i.e.
\[\forall f^{1(\Omega)}\left( f\in I\to \left\vert\INT{f}\right\vert\leq_\mathbb{R}\INT{\vert f\vert}\right).
\]
\end{enumerate}
\end{lemma}
\begin{proof}
\begin{enumerate}
\item If $\forall x\left( f(x)\leq g(x)\right)$, note that we have $(g-f)(x)\geq 0$ for all $x$. By axiom $(\int)_4$, we have thus that $\INT{(f-g)}\geq 0$ and therefore, we get $\INT{f}\leq\INT{g}$ by axiom $(\int)_2$.
\item This immediately follows from item (1).
\item By axiom $(I)_8$, we have $f\chi_{A^c},g\chi_{A^c}\in I$. As in particular $f\chi_{A^c}(x)\leq g\chi_{A^c}(x)$ holds for any $x$, we get
\[
\INT{f\chi_{A^c}}\leq\INT{g\chi_{A^c}}
\]
by item (1). Similar, as the axioms for the supremum norm imply that $(f-g)\chi_A(x)\leq \norm{f-g}_\infty\chi_A(x)$ holds for all $x$, item (1) together with axiom $(\int)_1$ implies
\[
\INT{(f-g)\chi_A}\leq\INT{\norm{f-g}_\infty\chi_A}=\norm{f-g}_\infty\PP(A)=0
\]
which yields
\[
\INT{f\chi_A}\leq\INT{g\chi_A}.
\]
As $f(x)= f\chi_A(x)+f\chi_{A^c}(x)$ holds for all $x$ (and similarly for $g$), we thus in particular get the claim using axiom $(\int)_2$.
\item This immediately follows from item (3).
\item Note that we have provably that
\[
-\vert f(x)\vert \leq f(x)\leq \vert f(x)\vert
\]
for any $x$ so that by item (1) and axiom $(\int)_2$, we have
\[
-\INT{\vert f\vert}\leq\INT{f}\leq\INT{\vert f\vert},
\]
i.e.\ that
\[
\left\vert\INT{f}\right\vert\leq\INT{\vert f\vert}.
\]
\end{enumerate}
\end{proof}

Note that by item (2) of the above lemma together with the axioms on the supremum norm $\norm{\cdot}_\infty$, we in particular have that $\INT{\cdot}$ is extensional w.r.t.\ $\norm{\cdot}_\infty$ in the sense that
\[
\forall f^{1(\Omega)},g^{1(\Omega)}\left( \norm{f-g}_\infty=_\mathbb{R}0\to \left\vert\INT{f}-\INT{g}\right\vert=_\mathbb{R}0\right).
\]

\section{A bound extraction theorem}\label{sec:metatheorem}
We now establish our main results, the bound extraction theorems, for the systems introduced previously. For that, as hinted on in the introduction, we follow the approach of the first metatheorems using abstract types presented in \cite{Koh2005} (see also \cite{GeK2008,Koh2008}).\footnote{As mentioned in the introduction already, this approach has been adapted to provide a treatment amenable to proof mining methods of a wide array of different mathematical notions in the past. We again refer to the references in the introduction for these results.} As the outline of our approach is rather standard in that way, we will sometimes only sketch the arguments instead of giving full detailed proofs, only spelling out those parts that are sensitive to the new ideas introduced in this paper. Throughout, to ease notation, we write $\mathcal{C}^\omega$ for the system $\mathcal{F}^\omega$ or one of its extensions as discussed previously.\\

As in the works mentioned above, the main tool for the metatheorems presented here is G\"odel's Dialectica interpretation \cite{Goe1958} which is combined with a negative translation by Kuroda \cite{Kur1951}. We recall the definitions of those central proof interpretations here:

\begin{definition}[\cite{Goe1958,Tro1973}]
The Dialectica interpretation $A^D=\exists\underline{x}\forall\underline{y} A_D(\underline{x},\underline{y})$ of a formula $A$ in the language of $\mathcal{C}^\omega$ (and its extensions) is defined via the following recursion on the structure of the formula:
\begin{enumerate}
\item $A^D:=A_D:=A$ for $A$ being a prime formula.
\end{enumerate}
If $A^D=\exists\underline{x}\forall\underline{y} A_D(\underline{x},\underline{y})$ and $B^D=\exists\underline{u}\forall\underline{v} B_D(\underline{u},\underline{v})$, we set
\begin{enumerate}
\setcounter{enumi}{1}
\item $(A\land B)^D:=\exists\underline{x},\underline{u}\forall\underline{y},\underline{v}(A\land B)_D$\\ where $(A\land B)_D(\underline{x},\underline{u},\underline{y},\underline{v}):=A_D(\underline{x},\underline{y})\land B_D(\underline{u},\underline{v})$,
\item $(A\lor B)^D:=\exists z^0,\underline{x},\underline{u}\forall\underline{y},\underline{v}(A\lor B)_D$\\ where $(A\lor B)_D(z^0,\underline{x},\underline{u},\underline{y},\underline{v}):=(z=0\rightarrow A_D(\underline{x},\underline{y}))\land (z\neq 0\rightarrow B_D(\underline{u},\underline{v}))$,
\item $(A\rightarrow B)^D:=\exists\underline{U},\underline{Y}\forall\underline{x},\underline{v}(A\rightarrow B)_D$\\ where $(A\rightarrow B)_D(\underline{U},\underline{Y},\underline{x},\underline{v}):=A_D(\underline{x},\underline{Y}\underline{x}\underline{v})\to B_D(\underline{U}\underline{x},\underline{v})$,
\item $(\exists z^\tau A(z))^D:=\exists z,\underline{x}\forall\underline{y}(\exists z^\tau A(z))_D$\\ where $(\exists z^\tau A(z))_D(z,\underline{x},\underline{y}):=A_D(\underline{x},\underline{y},z)$,
\item $(\forall z^\tau A(z))^D:=\exists\underline{X}\forall z,\underline{y}(\forall z^\tau A(z))_D$\\ where $(\forall z^\tau A(z))_D(\underline{X},z,\underline{y}):=A_D(\underline{X}z,\underline{y},z)$.
\end{enumerate}
\end{definition}

\begin{definition}[\cite{Kur1951}]
The negative translation of $A$ is defined by $A':=\neg\neg A^*$ where $A^*$ is defined by the following recursion on the structure of $A$:
\begin{enumerate}
\item $A^*:= A$ for prime $A$;
\item $(A\circ B)^*:= A^*\circ B^*$ for $\circ\in\{\land,\lor,\rightarrow\}$;
\item $(\exists x^\tau A)^*:=\exists x^\tau A^*$;
\item $(\forall x^\tau A)^*:=\forall x^\tau \neg\neg A^*$.
\end{enumerate}
\end{definition}

For the combination of these two interpretations, the following soundness result is one of the two central technical tools in the context of the proof of the proof mining metatheorems. In that context, we define $\mathcal{C}^{\omega-}$ as the system $\mathcal{C}^\omega$ \emph{without} the schemas $\QFAC$ and $\DC$.

\begin{lemma}[essentially \cite{Koh2005}]\label{lem:ndinterpretation}
Let $\mathcal{P}$ be a set of universal sentences and let $A(\underline{a})$ be an arbitrary formula (with only the variables $\underline{a}$ free) in the language of $\mathcal{F}^\omega$. Then the rule
\[
\begin{cases}\mathcal{F}^\omega+\mathcal{P}\vdash A(\underline{a})\Rightarrow\\
\mathcal{F}^{\omega-}+(\mathrm{BR})+\mathcal{P}\vdash\forall\underline{a},\underline{y}(A')_D(\underline{t}\underline{a},\underline{y},\underline{a})\end{cases}
\]
holds where $\underline{t}$ is a tuple of closed terms of $\mathcal{F}^{\omega-}+(\mathrm{BR})$ which can be extracted from the respective proof and $(\mathrm{BR})$ is the schema of \emph{simultaneous bar-recursion} of Spector \cite{Spe1962}, here extended to all types from $T^{\Omega,S}$ (similar as in e.g.\ \cite{Koh2008}).

This result extends to any suitable extension of the language of $\mathcal{F}^\omega$ (e.g.\ by any kind of new types and constants) together with any number of additional universal axioms in that language.
\end{lemma}

We omit the proof as it is almost exactly the same as the proof given for the analogous soundness result in \cite{Koh2005} (although this result from \cite{Koh2005} is of course not formulated for the system $\mathcal{F}^\omega$).\\

Besides the soundness of the Dialectica interpretation (together with the negative translation), the other one of the two central tools utilized in the metatheorems is that of majorizability. Originally introduced by Howard \cite{How1973}, the notion of majorizable functionals was later extended by Bezem \cite{Bez1985} to that of strongly majorizable functionals to provide a model for finite type arithmetic extended by the schema of bar recursion discussed before. In that way, this model of strongly majorizable functionals provides the crucial basis for proof mining metatheorems of systems allowing for dependent choice. In the context of the abstract types, we further need to consider an extension of this notion of strongly majorizable functionals to these new types. The first such extensions to abstract types have been devised in \cite{GeK2008,Koh2005}. However, in this setting (and essentially in all other settings for metatheorems proved afterwards), this extension was motivated and based on the metric structure assumed for the respective classes of spaces that were treated. We thus find ourselves here at a ``fork in the road'', so to say, where we have to extend the notion of majorizability sensibly to our types $\Omega$ and $S$, both representing spaces which do not carry any metric structure. The key insight, already mentioned and motivated throughout the previous sections many times (e.g.\ in the context of the admissibility of the axioms containing existential quantifiers over variables of types $\Omega$, $S$ or $S(0)$, etc.) is to 
\begin{enumerate}
\item majorize objects $A^S$ by natural numbers bounding the measure of $A$, 
\item majorize objects $x^\Omega$ by natural numbers bounding the measure of the full set $\Omega$ in $S$.
\end{enumerate}
In the case of a probability measure, any object of type $\Omega$ or $S$ is therefore uniformly majorized by $1$ but with the phrasing of (1) and (2), we wanted to highlight the general idea of this approach as it might be feasible also for more general finite contents.\\

In any way, similar to \cite{GeK2008,Koh2005}, the majorants for objects with types from $T^{\Omega,S}$ are objects with types from T according to the following extended projection:

\begin{definition}[essentially \cite{GeK2008}]
Define $\widehat{\tau}\in T$, given $\tau\in T^{\Omega,S}$, by recursion on the structure via
\[
\widehat{0}:=0,\;\widehat{\Omega}:=0,\;\widehat{S}:=0,\;\widehat{\tau(\xi)}:=\widehat{\tau}(\widehat{\xi}).
\]
\end{definition}

The majorizability relation $\gtrsim$ is then defined in tandem with the structure of all strongly majorizable functionals.

\begin{definition}[essentially \cite{GeK2008,Koh2005}]
Let $\Omega$ be a non-empty set, $S\subseteq 2^\Omega$ be an algebra and $\PP$ be a probability content on $S$. The structure $\mathcal{M}^{\omega,\Omega,S}$ and the majorizability relation $\gtrsim_\rho$ are defined by
\[
\begin{cases}
\mathcal{M}_0:=\mathbb{N}, n\gtrsim_0 m:=n\geq m\land n,m\in\mathbb{N},\\
\mathcal{M}_\Omega:= \Omega, n\gtrsim_\Omega x:= n\geq \PP(\Omega)\land n\in \mathcal{M}_0,x\in \mathcal{M}_\Omega,\\
\mathcal{M}_{S}:= S, n\gtrsim_{S} A:= n\geq \PP(A)\land n\in \mathcal{M}_0,A\in \mathcal{M}_{S},\\
f\gtrsim_{\tau(\xi)}x:=f\in \mathcal{M}_{\widehat{\tau}}^{\mathcal{M}_{\widehat{\xi}}}\land x\in \mathcal{M}_\tau^{\mathcal{M}_\xi}\\
\qquad\qquad\qquad\land\forall g\in \mathcal{M}_{\widehat{\xi}},y\in \mathcal{M}_\xi(g\gtrsim_\xi y\rightarrow fg\gtrsim_\tau xy)\\
\qquad\qquad\qquad\land\forall g,y\in \mathcal{M}_{\widehat{\xi}}(g\gtrsim_{\widehat{\xi}}y\rightarrow fg\gtrsim_{\widehat{\tau}}fy),\\
\mathcal{M}_{\tau(\xi)}:=\left\{x\in \mathcal{M}_\tau^{\mathcal{M}_\xi}\mid \exists f\in \mathcal{M}^{\mathcal{M}_{\widehat{\xi}}}_{\widehat{\tau}}:f\gtrsim_{\tau(\xi)}x\right\}.
\end{cases}
\]
\end{definition}

So, as already discussed previously, though only informally, as $\PP$ is a probability content, we have $1\gtrsim_S A$ for any $A\in S$ (as $\PP(A)\leq\PP(\Omega)=1$) as well as $1\gtrsim_\Omega x$ for any $x\in \Omega$ (but in the way the model is defined above, the definition immediately makes sense in the context of general finite contents).

Before we move on further, we now just quickly note that majorizability behaves nicely w.r.t.\ functions with multiple arguments as represented by their curryied variants.

\begin{lemma}[\cite{GeK2008,Koh2005}, see also Kohlenbach {\cite[Lemma 17.80]{Koh2008}}]\label{lem:majlemma}
Let $\xi=\tau(\xi_k)\dots(\xi_1)$. For $x^*:\mathcal{M}_{\widehat{\xi_1}}\to(\mathcal{M}_{\widehat{\xi_2}}\to\dots\to \mathcal{M}_{\widehat{\tau}})\dots)$ and $x:\mathcal{M}_{\xi_1}\to (\mathcal{M}_{\xi_2}\to\dots\to \mathcal{M}_\tau)\dots)$, we have $x^*\gtrsim_\xi x$ iff
\begin{enumerate}
\item[(a)] $\forall y_1^*,y_1,\dots,y_k^*,y_k\left(\bigwedge_{i=1}^k(y^*_i\gtrsim_{\xi_i}y_i)\rightarrow x^*y^*_1\dots y^*_k\gtrsim_\tau xy_1\dots y_k\right)$ and 
\item[(b)] $\forall y_1^*,y_1,\dots,y_k^*,y_k\left(\bigwedge_{i=1}^k(y^*_i\gtrsim_{\widehat{\xi_i}}y_i)\rightarrow x^*y^*_1\dots y^*_k\gtrsim_{\widehat{\tau}} x^*y_1\dots y_k\right)$.
\end{enumerate}
\end{lemma}

The other main structure featuring in the metatheorems is the structure of all set-theoretic functionals $\mathcal{S}^{\omega,\Omega,S}$, defined via $\mathcal{S}_0:=\mathbb{N}$, $\mathcal{S}_\Omega:= \Omega$, $\mathcal{S}_{S}:=S$ and
\[
\mathcal{S}_{\tau(\xi)}:=\mathcal{S}_{\tau}^{\mathcal{S}_{\xi}}.
\]
Both structures $\mathcal{S}^{\omega,\Omega,S}$ and $\mathcal{M}^{\omega,\Omega,S}$ later turn into models of our systems if equipped with corresponding interpretations for the respective additional constants, with $\mathcal{S}^{\omega,\Omega,S}$ serving as the structure for the intended standard models.\\

The proof of the bound extraction theorems now follows the following general high-level outline of most other such metatheorems: using functional interpretation and negative translation, one extracts realizers from (essentially) $\forall\exists$-theorems. These realizers have types from $T^{\Omega,S}$. We then use majorizability to construct bounds for these realizers, depending only on majorants of the parameters, which are validated in a model based on $\mathcal{M}^{\omega,\Omega,S}$. In a final step, we can then recover to the truth in a model based on the usual full set-theoretic structure $\mathcal{S}^{\omega,\Omega,S}$ if the types occurring in the axioms and the theorem are ``low enough'', which we will call admissible. Concretely, following \cite{GeK2008,Koh2005}, we introduce the following specific classes of types: We call a type $\xi$ \emph{of degree $n$} if $\xi\in T$ and it has degree $\leq n$ in the usual sense (see e.g.\ \cite{Koh2008}). Further we call $\xi$ \emph{small} if it is of the form $\xi=\xi_0(0)\dots(0)$ for $\xi_0\in\{0,\Omega,S\}$ (including $0,\Omega,S$) and call it \emph{admissible} if it is of the form $\xi=\xi_0(\tau_k)\dots(\tau_1)$ where each $\tau_i$ is small and $\xi_0\in\{0,\Omega,S\}$ (also including $0,\Omega,S$).\\

Further, also in analogy to \cite{GeK2008,Koh2005}, we define certain subclasses of formulas satisfying certain type restrictions: A formula $A$ is called a $\forall$-formula if $A=\forall\underline{a}^{\underline{\xi}} A_{qf}(\underline{a})$ with $A_{qf}$ quantifier-free and all types $\xi_i$ in $\underline{\xi}=(\xi_1,\dots,\xi_k)$ are admissible. A formula $A$ is called an $\exists$-formula if $A=\exists\underline{a}^{\underline{\xi}}A_{qf}(\underline{a})$ with similar $A_{qf}$ and $\underline{\xi}$.\\

The class $\Delta$ already mentioned previously, which was originally introduced in \cite{Koh1990,Koh1992b} (and then lifted to abstract types in \cite{GuK2016}) and signifies a class of commonly occurring formulas with trivial monotone functional interpretations, is now similarly introduced in the context of the systems studied in this paper: a formula of type $\Delta$ is any formula of the form
\[
\forall\underline{a}^{\underline{\delta}}\exists\underline{b}\leq_{\underline{\sigma}}\underline{r}\underline{a}\forall\underline{c}^{\underline{\gamma}}A_{qf}(\underline{a},\underline{b},\underline{c})
\]
where $A_{qf}$ is quantifier-free, the types in $\underline{\delta}$, $\underline{\sigma}$ and $\underline{\gamma}$ are admissible, $\underline{r}$ is a tuple of closed terms of appropriate type, $\leq$ is defined by recursion on the type via
\begin{enumerate}
\item $x\leq_0 y:=x\leq_0 y$,
\item $x\leq_\Omega y:=\PP(\Omega)\leq_\mathbb{R}\PP(\Omega)$,
\item $A\leq_{S} B:=\PP(A)\leq_\mathbb{R}\PP(B)$,
\item $x\leq_{\tau(\xi)} y:=\forall z^\xi(xz\leq_\tau yz)$,
\end{enumerate}
and $\underline{x}\leq_{\underline{\sigma}}\underline{y}$ is an abbreviation for $x_1\leq_{\sigma_1}y_1\land\dots\land x_k\leq_{\sigma_k}y_k$ where $\underline{x}$, $\underline{y}$ and $\underline{\sigma}$ are $k$-tuples of terms and types, respectively, such that $x_i$ and $y_i$ are of type $\sigma_i$.

Given a set $\beth$ of formulas of type $\Delta$, we write $\widetilde{\beth}$ for the set of all Skolem normal forms
\[
\exists\underline{B}\leq_{\underline{\sigma}(\underline{\delta})}\underline{r}\forall\underline{a}^{\underline{\delta}}\forall\underline{c}^{\underline{\gamma}}A_{qf}(\underline{a},\underline{B}\underline{a},\underline{c})
\]
for any $\forall\underline{a}^{\underline{\delta}}\exists\underline{b}\leq_{\underline{\sigma}}\underline{r}\underline{a}\forall\underline{c}^{\underline{\gamma}}A_{qf}(\underline{a},\underline{b},\underline{c})$ in $\beth$.

\begin{remark}
We want to note shortly that all axioms that were previously discussed as admissible based on our extended notion of majorizability can actually be seen as statements of type $\Delta$ in the context of the above definition. At first, the axiom $(\PP)_4$ can be equivalently rewritten as 
\[
\forall A^S,B^S\exists x^\Omega\leq_\Omega c_\Omega(\PP(A)>_\mathbb{R}\PP(B)\to \left( x\in A \land x\not\in B\right))
\]
and is thus immediately of type $\Delta$.\footnote{Note the importance of the constant $c_\Omega$ witnessing the non-emptiness of $\Omega$ for writing $(\PP)_4$ in that way.} Second, also the axiom $(I)_3$ can be rewritten with the additional boundedness information via
\begin{align*}
&\forall f^{1(\Omega)},a^1,b^1\exists A^S\leq_S \Omega\forall x^\Omega\left(f\in I\to \left(x\in f^{-1}([a,b])\leftrightarrow x\in A\right)\right)
\end{align*}
and thus is of type $\Delta$. Similarly, also the axiom $(I)_5$ can be rewritten as an axioms of type $\Delta$ as
\[
\forall f^{1(\Omega)},k^0\exists x^\Omega\leq_\Omega c_\Omega\left( f\in I\to \left( \norm{f}_\infty-2^{-k}\leq_\mathbb{R} \vert f(x)\vert\right)\right).
\]
Lastly, also the integrability axioms $(\int)_4$ and $(\int)_3$ can be rewritten as
\[
\forall f^{1(\Omega)}, k^0\exists x^\Omega\leq_\Omega c_\Omega\left( f\in I\land \INT{f}<_\mathbb{R} -2^{-k}\to \left( f(x)\leq_\mathbb{R} -2^{-k}\right)\right)
\]
and 
\begin{gather*}
\forall f^{1(\Omega)}, k^0\exists A^{S(0)}\leq_{S(0)}\lambda n^0.\Omega\\
\left(f\in I\to \INT{\left\vert f-\sum_{i=0}^{2^{k+1}b_f-1}\left(-([\norm{f}_\infty](0)+1)+\frac{i}{2^{k}}\right)\chi_{A(i)}\right\vert}\leq_\mathbb{R} 2^{-k}\right),
\end{gather*}
respectively, with $b_f:=[\norm{f}_\infty](0)+1$ as before, which turns them into axioms of type $\Delta$.
\end{remark}

Crucially, axioms of type $\Delta$ are trivialized under the monotone functional interpretation\footnote{While the interpretation was introduced under this name in \cite{Koh1996}, the idea of combining the Dialectica interpretation and majorization is already due to \cite{Koh1992b}.} and we treat any axiom of type $\Delta$ in $\mathcal{C}^\omega$ (or any suitable extension) ``in this spirit''. We here only write ``in this spirit'' as we actually do not use a monotone variant of the Dialectica interpretation but treat the functional interpretation part and the majorization part of the combined interpretation separately. In that way, we follow the approach given in \cite{GuK2016} (see also the recent \cite{Pis2023}) and treat axioms of type $\Delta$ by employing a construction that converts a theory with axioms of such a type into a theory using only additional purely universal axioms formulated using the Skolem functions of these axioms. This new theory is then used in combination with the functional interpretation to extract the respective terms and then the proof proceeds as outlined before.

Concretely, we now proceed as follows: Let $\beth$ be a set of axioms of type $\Delta$ and write $\widehat{\mathcal{C}}^\omega$ for $\mathcal{C}^\omega$ without any of its axioms of type $\Delta$. Then, we form a new theory $\overline{\mathcal{C}}^\omega_\beth$ from $\widehat{\mathcal{C}}^\omega$ by adding the Skolem functionals $\underline{B}$ of any axiom of type $\Delta$ of $\mathcal{C}^\omega+\beth$, say of the form
\[
\forall\underline{a}^{\underline{\delta}}\exists\underline{b}\leq_{\underline{\sigma}}\underline{r}\underline{a}\forall\underline{c}^{\underline{\gamma}}A_{qf}(\underline{a},\underline{b},\underline{c}),
\]
as new constants to the language and simultaneously adding the corresponding ``instantiated Skolem normal form'', i.e.
\[
\underline{B}\leq_{\underline{\sigma}(\underline{\delta})}\underline{r}\land \forall\underline{a}^{\underline{\delta}}\forall\underline{c}^{\underline{\gamma}}A_{qf}(\underline{a},\underline{B}\underline{a},\underline{c}),
\]
as a new axiom. Therefore, the system $\overline{\mathcal{C}}^\omega_\beth$ only extends $\mathcal{F}^\omega$ by new types, constants and universal axioms. Therefore, as mentioned before, Lemma \ref{lem:ndinterpretation} also applies to this system where the conclusion is then proved in $\overline{\mathcal{C}}^{\omega-}_\beth+(\mathrm{BR})$, i.e.\ $\overline{\mathcal{C}}^{\omega}_\beth$ with the principles $\QFAC$ and $\DC$ removed and where the scheme of simultaneous bar-recursion is added.

In the case where the extension $\mathcal{C}^\omega$ contains the rule $(\bigcup)_2$, we for simplicity assume that in the process of forming the extended theory, this rule is also removed in the sense that $\overline{\mathcal{C}}^{\omega}_\beth$ does not contain the rule and for any provable premise
\[
\mathcal{C}^\omega+\beth\vdash F_{qf}\to \forall n^0\left( A(n)\subseteq_SB\right),
\]
we add the corresponding conclusion
\[
F_{qf}\to \bigcup A\subseteq_S B
\]
as an axiom of $\overline{\mathcal{C}}^{\omega}_\beth$.\\

We now move on to the central result on the majorization part of the chosen approach to the bound extraction theorems, stating that every closed term in the underlying language of the system in question is majorizable. As such, the result is similar to Lemma 9.11 in \cite{GeK2008}.

\begin{lemma}\label{lem:majmainresult}
Let $\mathcal{C}^\omega$ be (one of the previously discussed extensions of) the system $\mathcal{F}^\omega[\PP]$ and let $\beth$ be a set of additional axioms of type $\Delta$. Let $\Omega$ be a non-empty set and let $S\subseteq 2^\Omega$ be an algebra (or, in the context of the constant $\bigcup$, a $\sigma$-algebra). Let $\PP$ be a probability content on $S$. Then $\mathcal{M}^{\omega,\Omega,S}$ is a model of $\overline{\mathcal{C}}^{\omega-}_\beth+(\mathrm{BR})$, provided $\mathcal{S}^{\omega,\Omega,S}\models \beth$ (with $\mathcal{M}^{\omega,\Omega,S}$ and $\mathcal{S}^{\omega,\Omega,S}$ defined via suitable interpretations of the additional constants in $\mathcal{C}^\omega$). Moreover, for any closed term $t$ of $\overline{\mathcal{C}}^{\omega-}_\beth+(\mathrm{BR})$, one can construct a closed term $t^*$ of $\mathcal{A}^\omega+(\mathrm{BR})$ such that
\[
\mathcal{M}^{\omega,\Omega,S}\models\left( t^*\gtrsim t\right).
\]
\end{lemma}
\begin{proof}
The structure of the proof is standard and similar to proofs of related results from the literature (see e.g.\ \cite{Koh2008}). As such, we only discuss the interpretations of the new constants added to $\mathcal{A}^\omega$ to form the respective theories as well as their majorizations. For the constants already contained in $\mathcal{A}^\omega$, we may choose suitable interpretations as in \cite{Koh2008} and for majorizing a composition of terms, we may similarly proceed as outlined therein. For that, we now first focus on $\mathcal{F}^\omega[\PP]$ and assume that there are no further axioms of type $\Delta$ beyond those already contained in $\mathcal{F}^\omega[\PP]$. For any $\mathcal{C}^\omega$, we deal with any set $\beth$ of additional axioms of type $\Delta$ and the respectively induced constants later on by moving to the theory $\overline{\mathcal{C}}^\omega_\beth$.\\

Now, for the new constants added to $\mathcal{A}^\omega$ to form $\mathcal{F}^\omega[\PP]$, we consider the following interpretations (writing $\mathcal{M}$ for $\mathcal{M}^{\omega,\Omega,S}$): 
\begin{itemize}
\item $[\mathrm{eq}]_{\mathcal{M}}:=\text{the characteristic function of the equaliy relation in }\Omega$;
\item $[\in]_{\mathcal{M}}:=\text{the characteristic function of the element relation in }S$;
\item $[\cup]_\mathcal{M}:=\text{union in }S$;
\item $[(\cdot)^c]_\mathcal{M}:=\text{complement in }S$;
\item $[\emptyset]_\mathcal{M}:=\text{the empty set in }S$;
\item $[\PP]_\mathcal{M}:= \lambda A^S.(\PP(A))_\circ$ where $\PP$ is the content fixed in the context of this lemma.
\end{itemize}

This is only well-defined in $\mathcal{M}^{\omega,\Omega,S}$ if we can construct majorants of these objects. This we can do as follows:
\begin{itemize}
\item $\lambda x^0,y^0.1\gtrsim\mathrm{eq}$;
\item $\lambda x^0,y^0.1\gtrsim\, \in$;
\item $\lambda x^0,y^0.1\gtrsim \cup$;
\item $\lambda x^0. 1\gtrsim (\cdot)^c$;
\item $0\gtrsim \emptyset$;
\item $\lambda x^0.(x)_\circ\gtrsim \PP$.
\end{itemize}
Note that in the last item, the operation $(x)_\circ$ is definable in $\mathcal{A}^\omega$ via a closed term as $x$ is of type $0$.

For justifying that those terms really are majorants of the respective constants, we argue as follows: The first four items immediately follow from the fact that $\PP(A)\leq \PP(\Omega)=1$ (i.e.\ that $\PP$ is a probability content) and that $\PP(\emptyset)=_\mathbb{R}0$. The last item follows immediately from Lemma \ref{lem:circprop} as clearly, if $x\geq_\mathbb{R}\PP(A)$, then $(x)_\circ\gtrsim (\PP(A))_\circ$.\\

In the case where $\mathcal{C}^\omega$ contains the respective additional constant $\bigcup$, a corresponding interpretation is naturally defined by
\begin{itemize}
\item $[\bigcup]_\mathcal{M}:=\text{countably infinite union in S}$,
\end{itemize}
which is well-defined since we in this context assume that $S$ is a $\sigma$-algebra. We can achieve majorization as before by exploiting that $\PP$ is a finite content with
\begin{itemize}
\item $\lambda f^{0(0)}.1\gtrsim\bigcup$.
\end{itemize}

\medskip

Lastly, if the system $\mathcal{C}^\omega$ contains the respective constants and axioms for treating integrals, we choose corresponding interpretations of the additional constants as follows:
\begin{itemize}
\item $\big[[\cdot,\cdot]\big]_\mathcal{M}:=\lambda a^1,b^1,x^1.\begin{cases}0&\text{if }r_x\in [r_a,r_b];\\1&\text{otherwise};\end{cases}$
\item $[(\cdot)^{-1}]_\mathcal{M}:=\lambda f^{1(\Omega)},A^{0(1)},x^\Omega.\begin{cases}0&\text{if }x\in f^{-1}(\{r_a\mid a^1: A(a)=_00\});\\1&\text{otherwise};\end{cases}$
\item $[I]_\mathcal{M}:=$ the characteristic function of a set of bounded and weakly Borel-measurable functions $f^{1(\Omega)}$ which is closed under linear combinations, multiplication with characteristic functions and absolute values;
\item $[\norm{\cdot}_\infty]_\mathcal{M}:=\lambda f^{1(\Omega)}.\begin{cases}(\sup_{x\in \Omega}\vert f(x)\vert )_\circ&\text{if }f\text{ is bounded};\\0&\text{otherwise};\end{cases}$
\item $[\INT{\cdot}]_\mathcal{M}:=\lambda f^{1(\Omega)}.\begin{cases}(\INT{f})_\circ&\text{if }f\text{ is bounded and weakly Borel-measurable};\\0&\text{otherwise};\end{cases}$\\
where the latter $\INT{f}$ represents the usual integral defined over the content.
\end{itemize}

Note that here, we now rely on the extended operator $(\cdot)_\circ$ operating on all real numbers as the integral of a general integrable function may be negative. As for majorization, we rely on the following constructions:
\begin{itemize}
\item $\lambda a^1,b^1,r^1.1\gtrsim [\cdot,\cdot]$;
\item $\lambda f^{1(0)},A^{0(1)},x^0.1\gtrsim (\cdot)^{-1}$;
\item $\lambda f^{1(0)}.1\gtrsim I$;
\item $\lambda f^{1(0)}.(f(1)(0)+1)_\circ\gtrsim \norm{\cdot}_\infty$;
\item $\lambda f^{1(0)}.(f(1)(0)+1)_\circ\gtrsim \INT{\cdot}$.
\end{itemize}
The first three items are immediate as we deal with characteristic functions. For the fourth, note that if ${f^*}^{1(0)}\gtrsim f^{1(\Omega)}$, then
\[
\forall n^0,x^\Omega\left( n\gtrsim x\to f^*(n)\gtrsim f(x)\right)
\]
and as $1\gtrsim x$ for any $x^\Omega$ as $1=\PP(\Omega)$, we have $f^*(1)\gtrsim f(x)$ for any $x^\Omega$. Therefore, we have $f^*(1)(0)+1\geq_\mathbb{R} f(x)$ for any $x^\Omega$ and thus $f^*(1)(0)+1\geq_\mathbb{R} \norm{f}_\infty$ so that the result follows from Lemma \ref{lem:circprop}. 

Lastly, note that for any bounded and weakly Borel-measurable function $f$, we have that $\vert\INT{f}\vert\leq_\mathbb{R}\INT{\vert f\vert}$ so that
\[
\left\vert \INT{f}\right\vert\leq_\mathbb{R}\INT{\vert f\vert}\leq_\mathbb{R} \norm{\vert f\vert}_\infty=_\mathbb{R}\norm{f}_\infty\leq_\mathbb{R} f^*(1)(0)+1
\]
as before. The majorizability result then follows again from Lemma \ref{lem:circprop}.\\

That $\mathcal{M}^{\omega,\Omega,S}$ with these chosen interpretations is a model of $\mathcal{C}^{\omega-}+(\mathrm{BR})$ can be shown similarly as in analogous results (see e.g.\ \cite{Koh2008}). The intended interpretations of the constants of $\mathcal{C}^\omega$ in $\mathcal{S}^{\omega,\Omega,S}$, turning $\mathcal{S}^{\omega,\Omega,S}$ into a model of these systems, are defined in analogy to the corresponding model $\mathcal{M}^{\omega,\Omega,S}$ defined above.\\

For treating the other additional axioms in $\mathcal{C}^\omega+\beth$ of type $\Delta$ beyond the axioms already contained in $\mathcal{C}^\omega$, we rely on the following argument (akin to \cite{GuK2016}, Lemma 5.11) showing that $\mathcal{S}^{\omega,\Omega,S}\models \beth$ implies $\mathcal{M}^{\omega,\Omega,S}\models\widetilde{\beth}$. For this, the proof given in \cite{GuK2016} for Lemma 5.11 carries over which we sketch here: While $\mathcal{M}^{\omega,\Omega,S}$ in general is not a model of the axiom of choice \cite{Koh1992a}, one can show (similar to \cite{Koh1992a}) that $\mathcal{M}^{\omega,\Omega,S}\models \mathsf{b}\text{-}\mathsf{AC}_{\Omega,S}$ where
\[
\mathsf{b}\text{-}\mathsf{AC}_{\Omega,S}:=\bigcup_{\delta,\rho\in T^{\Omega,S}}\mathsf{b}\text{-}\mathsf{AC}^{\delta,\rho}
\]
with
\[
\mathsf{b}\text{-}\mathsf{AC}^{\delta,\rho}:=\forall Z^{\rho(\delta)}\left( \forall x^\delta\exists y\leq_\rho Zx A(x,y,Z)\to \exists Y\leq_{\rho(\delta)} Z\forall x^\delta A(x,Yx,Z)\right).
\]
Now, for small types $\rho$, we have $M_\rho=S_\rho$ while for admissible types $\rho$, we have $M_\rho\subseteq S_\rho$ (for which it is important that admissible types take arguments of small types). For this, the proof given in \cite{GeK2008} carries over. Further, we need that it is provable in $\mathcal{C}^{\omega-}$ that
\[
\forall x',x,y\left( x'\gtrsim_\rho x\land x\geq_\rho y\to x'\gtrsim_\rho y\right) \tag{$\dagger$}
\]
holds for all types $\rho$ which can be shown similar as e.g.\ in \cite{Koh2008}.

Suppose now that 
\[
\mathcal{S}^{\omega,\Omega,S}\models \forall\underline{a}^{\underline{\delta}}\exists\underline{b}\leq_{\underline{\sigma}}\underline{r}\underline{a}\forall\underline{c}^{\underline{\gamma}}A_{qf}(\underline{a},\underline{b},\underline{c}).
\]
Then also $\mathcal{M}^{\omega,\Omega,S}$ is a model of this sentence: First the types of the variables which are universally quantified are admissible, so over $\mathcal{M}^{\omega,\Omega,S}$ the domain of the universal quantifiers is reduced. For the witnesses for $\underline{b}$, which exist in $\mathcal{S}^{\omega,\Omega,S}$, note first that these could potentially live in $\mathcal{M}^{\omega,\Omega,S}$ as the types of the variables in $\underline{b}$ are admissible, i.e.\ they take arguments of small types and map into small types. It thus only remains to be seen whether such a witness is majorizable for majorizable inputs. However, by the above argument, the terms in $\underline{r}$ are all majorizable and if $\underline{a}$ comes from $\mathcal{M}^{\omega,\Omega,S}$, then $\underline{r}\underline{a}$ is majorizable. That we have $\underline{b}\leq_{\underline{\sigma}}\underline{r}\underline{a}$ in $\mathcal{M}^{\omega,\Omega,S}$ now implies that $\underline{b}$ is majorizable by $(\dagger)$ (and consequently the corresponding interpretations exist in $\mathcal{M}^{\omega,\Omega,S}$ too). Lastly, it is rather immediate to see that $\mathcal{M}^{\omega,\Omega,S}\models \beth$ implies $\mathcal{M}^{\omega,\Omega,S}\models\widetilde{\beth}$ using $\mathsf{b}\text{-}\mathsf{AC}_{\Omega,S}$.\\

From $\mathcal{M}^{\omega,\Omega,S}\models\widetilde{\beth}$, we immediately get that the corresponding Skolem functions have interpretations in $\mathcal{M}^{\omega,\Omega,S}$, that the corresponding structures defined by some canonical interpretations of those additional constants are indeed models of those variants of the systems where the corresponding Skolem functionals of these axioms are added and where the axioms themselves are replaced by their instantiated Skolem normal forms (i.e.\ $\overline{\mathcal{C}}^{\omega-}_\beth$ and its extensions) and, lastly, that the above majorizability result extends to these systems.\\

Note that, technically, these arguments were already needed in the above considerations to see that $\mathcal{M}^{\omega,\Omega,S}$ really is a model of $\mathcal{C}^{\omega-}+(\mathrm{BR})$. However, we did not discuss this there explicitly as for those specific axioms of type $\Delta$ belonging to $\mathcal{C}^{\omega-}+(\mathrm{BR})$, the types of the variables occurring in them are not only small but actually all among $\{0,1,\Omega,S,S(0)\}$ so that it was immediately clear that the models coincide at that level (essentially just by definition) and we thus omitted such a general discussion there.
\end{proof}

Combined with the Dialectica interpretation, the main result we then arrive at is the following bound extraction result for classical proofs:

\begin{theorem}\label{thm:metatheorem}
Let $\mathcal{C}^\omega$ be (one of the previously discussed extensions of) the system $\mathcal{F}^\omega[\PP]$ and let $\beth$ be a set of formulas of type $\Delta$. Let $\tau$ be admissible, $\delta$ be of degree $1$ and $s$ be a closed term of $\mathcal{C}^\omega$ of type $\sigma(\delta)$ for admissible $\sigma$ and let $B_\forall(x,y,z,u)$/$C_\exists(x,y,z,v)$ be $\forall$-/$\exists$-formulas of $\mathcal{C}^{\omega}$ with only $x,y,z,u$/$x,y,z,v$ free. If
\[
\mathcal{C}^{\omega}+\beth\vdash\forall x^\delta\forall y\leq_\sigma s(x)\forall z^\tau\left(\forall u^0 B_\forall(x,y,z,u)\rightarrow\exists v^0 C_\exists(x,y,z,v)\right),
\]
then one can extract a partial functional $\Phi:\mathcal{S}_{\delta}\times \mathcal{S}_{\widehat{\tau}}\rightharpoonup\mathbb{N}$ which is total and (bar-recursively) computable on $\mathcal{M}_\delta\times \mathcal{M}_{\widehat{\tau}}$ and such that for all $x\in \mathcal{S}_\delta$, $z\in \mathcal{S}_\tau$, $z^*\in \mathcal{S}_{\widehat{\tau}}$, if $z^*\gtrsim z$, then
\[
\mathcal{S}^{\omega,\Omega,S}\models\forall y\leq_\sigma s(x)\left(\forall u\leq_0\Phi(x,z^*) B_\forall(x,y,z,u)\rightarrow\exists v\leq_0\Phi(x,z^*)C_\exists(x,y,z,v)\right)
\]
holds whenever $\mathcal{S}^{\omega,\Omega,S}\models \beth$ for $\mathcal{S}^{\omega,\Omega,S}$ defined via any non-empty set $\Omega$ and any algebra $S\subseteq 2^\Omega$ (or, in the context of the constant $\bigcup$, any $\sigma$-algebra) together with any probability content $\PP$ on $S$ (and with suitable interpretations of the additional constants). Further:
\begin{enumerate}
\item If $\widehat{\tau}$ is of degree $1$, then $\Phi$ is a total computable functional. 
\item We may have tuples instead of single variables $x,y,z,u,v$ and a finite conjunction instead of a single premise $\forall u^0 B_\forall(x,y,z,u)$.
\item If the claim is proved without $\DC$, then $\tau$ may be arbitrary and $\Phi$ will be a total functional on $\mathcal{S}_\delta\times \mathcal{S}_{\widehat{\tau}}$ which is primitive recursive in the sense of G\"odel \cite{Goe1958} and Hilbert \cite{Hil1926}. In that case, also plain majorization can be used instead of strong majorization \emph{(}see e.g.\ \cite{Koh2008}\emph{)}.
\end{enumerate}
\end{theorem}
\begin{proof}
First, assume that $\mathcal{S}^{\omega,\Omega,S}\models \beth$ and (for simplicity) that
\[
\mathcal{C}^{\omega}+\beth\vdash\forall z^\tau\left(\forall u^0 B_\forall(z,u)\rightarrow\exists v^0 C_\exists(z,v)\right).
\]
Clearly, the same statement is then also provable in $\overline{\mathcal{C}}^\omega_\beth$. By assumption, $B_\forall(z,u)=\forall\underline{a} B_{qf}(z,u,\underline{a})$ and $C_\exists(z,v)=\exists\underline{b} C_{qf}(z,v,\underline{b})$ for quantifier-free $B_{qf}$ and $C_{qf}$. Thus, by prenexiation, we get
\[
\overline{\mathcal{C}}^\omega_\beth\vdash\forall z^\tau\exists u,v,\underline{a},\underline{b}(B_{qf}(z,u,\underline{a})\rightarrow C_{qf}(z,v,\underline{b})).
\]
Using Lemma \ref{lem:ndinterpretation} (which is applicable as $\overline{\mathcal{C}}^\omega_\beth$ is an extension of $\mathcal{F}^\omega$ only by new constants and purely universal axioms) and disregarding the realizers for $\underline{a},\underline{b}$, we get closed terms $t_u,t_v$ of $\overline{\mathcal{C}}^{\omega-}_\beth+(\mathrm{BR})$ such that
\[
\overline{\mathcal{C}}^{\omega-}_\beth+(\mathrm{BR})\vdash\forall z^\tau(B_\forall (z,t_u(z))\rightarrow C_\exists(z,t_v(z))).
\]
By Lemma \ref{lem:majmainresult} there are closed terms $t^*_u,t^*_v$ of $\mathcal{A}^\omega+(\mathrm{BR})$ such that
\[
\mathcal{M}^{\omega,\Omega,S}\models t^*_u\gtrsim t_u\land t^*_v\gtrsim t_v\land\forall z^\tau(B_\forall(z,t_u(z))\rightarrow C_\exists(z,t_v(z)))
\]
for all non-empty sets $\Omega$, any algebra (or $\sigma$-algebra) $S\subseteq 2^\Omega$ and any probability content $\PP$ on $S$ and where the constants are interpreted as in Lemma \ref{lem:majmainresult}. Define 
\[
\Phi(z^*):=\max\{t^*_u(z^*),t^*_v(z^*)\}.
\]
Then
\[
\mathcal{M}^{\omega,\Omega,S}\models\forall u\leq_0\Phi(z^*) B_\forall(z,u)\rightarrow\exists v\leq_0\Phi(z^*) C_\exists(z,v)
\]
holds for all $z\in \mathcal{M}_\tau$ and $z^*\in \mathcal{M}_{\widehat{\tau}}$ with $z^*\gtrsim z$. The conclusion that $\mathcal{S}^{\omega,\Omega,S}$ satisfies the same sentence can be achieved as in the proof of Theorem 17.52 in \cite{Koh2008} which we sketch here: Note that in the conclusion, we restrict ourselves to those $z$ which have majorants $z^*$. As the type of $z$ is admissible, it takes arguments of small type for which $\mathcal{M}^{\omega,\Omega,S}$ and $\mathcal{S}^{\omega,\Omega,S}$ coincide (recall the proof of Lemma \ref{lem:majmainresult}). Therefore, any such $z,z^*$ from $\mathcal{S}^{\omega,\Omega,S}$ also live in $\mathcal{M}^{\omega,\Omega,S}$ so that $\Phi(z^*)$ is well-defined for $z,z^*$ belonging to $\mathcal{S}^{\omega,\Omega,S}$ with $z^*\gtrsim z$. In $B_\forall$, all types are admissible to that truth in $\mathcal{S}^{\omega,\Omega,S}$ implies truth in $\mathcal{M}^{\omega,\Omega,S}$ and similarly for $C_\exists$ where thus truth in $\mathcal{M}^{\omega,\Omega,S}$ implies truth in $\mathcal{S}^{\omega,\Omega,S}$. Lastly, as in Lemma 17.84 in \cite{Koh2008}, we can show that as $\Phi$ is of type $0(\widehat{\tau})$, the interpretations of $\Phi$ in $\mathcal{S}^{\omega,\Omega,S}$ and $\mathcal{M}^{\omega,\Omega,S}$ coincide on majorizable elements. All in all we have that
\[
\mathcal{S}^{\omega,\Omega,S}\models\forall u\leq_0\Phi(z^*) B_\forall(z,u)\rightarrow\exists v\leq_0\Phi(z^*) C_\exists(z,v)
\]
holds for all $z\in S_\tau$ and $z^*\in S_{\widehat{\tau}}$ with $z^*\gtrsim z$.

\smallskip

The additional $\forall x^\delta\forall y\leq_\sigma s(x)$ can be treated as e.g.\ discussed in \cite{Koh2005} and we thus omit any details. Similarly, item (1) can be shown as in the proof of Theorem 17.52 from \cite{Koh2008} (see page 428 therein). Further, (2) is immediate and (3) follows from the fact that without $\DC$, bar recursion becomes superfluous and the model $\mathcal{M}^{\omega,\Omega,S}$ can be avoided.
\end{proof}

\section{Applications of the metatheorems}
\label{sec:appliations}
In this section, we are now concerned with the applications of the above metatheorems. Concretely, we want to indicate how the systems introduced prior can be used together with their metatheorems to recognize previous (ad hoc) applications in the spirit of the proof mining program as proper instances of proof-theoretic bound extraction theorems. For this, we here focus on the seminal work \cite{ADR2012} by Avigad, Dean and Rute where we find that the quantitative results obtained in \cite{ADR2012} for Egorov's theorem as well as the dominated convergence theorem, although in general being of bar-recursive complexity (which is due to the use of a certain principle of countable choice as we will later see formally), are nevertheless highly uniform, being in particular independent of the space, the measurable sets and the measure. As already mentioned in the introduction, the authors of \cite{ADR2012} presumed that this independence can be explained using some instantiation of the notion of majorizability. In that way, the metatheorems proved before and the discussion in this section show that this intuition was correct and that the uniformities are a necessary consequence of the novel form of majorizability introduced in this paper.\\

However, as already discussed in the introduction, we want to mention that the focus on the work \cite{ADR2012} shall be understood to be merely indicative on the usefulness of the systems and metatheorems discussed before and that essentially all other quantitative works on probability theory in the spirit of proof mining that have so far been considered in the literature (as well as forthcoming work by the authors together with Thomas Powell) can be similarly recognized as instances of the metatheorems proved here and, similarly, also there the peculiar uniformities observed in practice are a priori guaranteed by the approach towards the metatheorems chosen here. In particular, as discussed in the introduction, the logical perspective provided by this work was crucial for obtaining the recent case studies presented in \cite{Ner2024,NP2024}.\\

Now, the work \cite{ADR2012} is concretely concerned with interrelations between different modes of convergence for sequences of random variables. The most prolific of these modes, also based on its similarity to a usual notion of pointwise convergence of functions, is that of almost sure convergence.

\begin{definition}[Almost sure convergence]\label{def:almostSureConv}
Let $(\Omega,S,\PP)$ be a probability space and $(X_n)$ be a sequence of random variables $X_n:\Omega\to\mathbb{R}$ (i.e.\ $X_n$ is measurable w.r.t.\ $S$ and the Borel $\sigma$-algebra on $\mathbb{R}$). Then $(X_n)$ is said to converge almost surely to a random variable $X:\Omega\to\mathbb{R}$ if
\[
\PP(\{x\in \Omega\mid X_n(x) \to X(x)\}) = 1.
\]
\end{definition}

This notion of almost sure convergence does not lend itself easily for a quantitative account of that convergence. Thus, in many cases where probability theorists are concerned with quantitative results (see e.g.\ \cite{Luz2018,Sie1975}), they opt for a different, but equivalent, formulation of almost sure convergence known as almost uniform convergence.

\begin{definition}[Almost uniform convergence]\label{def:infau}
Let $(\Omega,S,\PP)$ be a probability space and $(X_n)$ be a sequence of random variables $X_n:\Omega\to\mathbb{R}$. Then $(X_n)$ is said to converge almost uniformly\footnote{Almost uniform convergence is usually formulated by requiring that for every $\varepsilon>0$, there exists a measurable set $F_\varepsilon\in S$ with $\PP(F_\varepsilon) \le \varepsilon$ such that $(X_n)$ converges uniformly on $F_\varepsilon^c$ to $X$. The (equivalent) formulation given here is however more fruitful from an applied proof-theoretic perspective as it immediately allows for a very natural Cauchy-type variant.} to a random variable $X:\Omega\to\mathbb{R}$ if for all $\varepsilon, \delta >0$ there exists an $N \in \NN$ such that 
\[
\PP(\{x\in \Omega\mid \forall n \ge N \left( |X_n(x) - X(x)| \le \varepsilon\right)\}) > 1 - \delta.
\]
\end{definition}

The seminal result that these two notions of convergence are indeed equivalent is known as Egorov's theorem \cite{Ego1911}. Note that since $(\Omega,S,\PP)$ is a probability space and the $X_n$'s are random variables (and therefore measurable), the sets we take the probability of in the above definitions are measurable sets.\\

This result was analyzed quantitatively in \cite{ADR2012} and then was used in turn also to provide a quantitative dominated convergence theorem for the Lebesgue integral. With the results from this section, we will be able to recognize this analysis as an instance of the preceding metatheorems for probability contents on algebras.\\

We now move towards formalizing the results from \cite{ADR2012} and for that introduce a sequence of random variables into the system $\mathcal{F}^\omega[\PP,\mathrm{Integral}]$ by adding an additional constant $X^{1(\Omega)(0)}$ together with the axiom
\[
\forall n^0\left( X(n)\in I\right)
\]
to stipulate that the sequence in question belongs to our subspace of bounded and weakly Borel-measurable functions. It is clear that Theorem \ref{thm:metatheorem} extends from $\mathcal{F}^\omega[\PP,\mathrm{Integral}]$ to its extension by this constant $X$ as any $X(n)$ is trivially majorizable as it is bounded and the whole constant $X$ is thus majorized by a maximum construction. In that system, using axiom $(I)_3$ that asserts the weak Borel-measurability of any $X(n)$, it is then in particular a consequence of $\Pi^\Omega_1$-$\mathsf{AC}$ (and actually of $\mathsf{b}\text{-}\mathsf{AC}_{\Omega,S}$ by regarding quantification over the type $S$ as bounded) that there exists a functional $P^{S(0)(0)(0)}$ such that
\[
\forall a^0,b^0,c^0,x^\Omega(x \in P(a,b,c) \leftrightarrow x\in (\vert X(c)-X(b)\vert)^{-1}([0,2^{-a}])).
\]
We add such a $P$ directly into the language of the system via a new constant of the appropriate type together with the above defining property as an axiom and denote the resulting system by $\mathcal{F}^\omega[\PP,\mathrm{Integral},X]$.\\

Using this functional $P$, we can then formally introduce the alternative way of formulating almost uniform convergence using finite unions as (implicitly) introduced in \cite{ADR2012}, which allows for both a natural metastable variant as well as to extended any discussion regarding this notion naturally to the context of contents:

\begin{definition}\label{def:au}
We say that $X$ converges almost uniformly with respect to finite unions if 
\[
\forall k^0, a^0\exists b^0\forall c^0\left( \PP\left(\bigcup_{i=b}^c\bigcup_{j=b}^c P(a,i,j)^c\right)\le_\mathbb{R} 2^{-k}\right).
\]
\end{definition}

It is rather immediately clear that this notion of almost uniform convergence w.r.t.\ finite unions is equivalent over probability spaces to the usual notion of almost uniform convergence. Further, we want to emphasize that this mode was not explicitly introduced in \cite{ADR2012} but rather \emph{implicitly} as already mentioned above as it is naturally suggested through the quantitative rendering of almost uniform convergence used in \cite{ADR2012} in their main quantitative result given in Theorem 3.1. Concretely, one immediately finds that a solution of the monotone functional interpretation of the negative translation of almost uniform convergence w.r.t.\ finite unions is exactly a function $M(k,a)$ providing a $2^{-k}$-uniform bound for the $2^{-a}$-metastable convergence of the sequence coded by $X$ as introduced in \cite{ADR2012}.\\

Similarly, we can also give a formal representation of the notion of almost uniform metastable pointwise convergence as introduced in \cite{ADR2012} in the context of the system $\mathcal{F}^\omega[\PP,\mathrm{Integral},X]$:

\begin{definition}\label{def:aump}
We say that $X$ converges almost uniform metastable pointwisely if
\[
\forall k^0,a^0,F^1\exists b^0\left( \PP\left(\bigcap_{m=0}^b\bigcup_{i=m}^{F(m)}\bigcup_{j=m}^{F(m)}P(a,i,j)^c\right)\le_\mathbb{R} 2^{-k}\right).
\]
\end{definition}

Contrary to Definition \ref{def:au}, this mode of convergence was explicitly introduced in \cite{ADR2012} as a ``more quantitatively friendly'' version of the notion of almost sure convergence. In particular, as shown by Proposition 4.1 in \cite{ADR2012}, the notion of almost uniform metastable pointwise convergence coincides over probability spaces with that of almost sure convergence. Also, as essentially already observed in \cite{ADR2012}, a solution to the monotone functional interpretation (of the negative translation of) the statement of almost uniform metastable pointwise convergence is exactly a function $M'(k,a)$ providing a $2^{-k}$-uniform bound on the $2^{-a}$-metastable pointwise convergence of the sequence coded by $X$ as introduced in \cite{ADR2012}.\\

In \cite{ADR2012}, the authors now provide a quantitative version of Egorov's theorem by constructing a solution of the monotone functional interpretation of the negative translation of almost uniform convergence w.r.t.\ finite unions from a solution of the monotone functional interpretation (of the negative translation of) the statement of almost uniform metastable pointwise convergence and so they, in particular, provide a uniform quantitative rendering of the corresponding implication. We justify this application by showing in the following that already the system $\mathcal{F}^\omega[\PP,\mathrm{Integral},X]$ proves this implication between the above two modes of convergence. Besides thereby explaining the success and the uniformities of the quantitative version of Egorov's theorem from \cite{ADR2012}, our formal investigations here show in particular that the corresponding results from \cite{ADR2012} are true for probability contents and not just probability measures as illustrated in \cite{ADR2012} (by virtue of the implication from almost uniform metastable pointwise convergence to almost uniform convergence w.r.t.\ finite unions being provable in the system $\mathcal{F}^\omega[\PP,\mathrm{Integral},X]$), i.e. the authors inadvertently provided an ``Egorov-like theorem'' for probability contents. Concretely, this seems to be in particular due to the above renderings of the notions of almost sure and almost uniform convergence introduced via a finitary perspective informed by proof mining in \cite{ADR2012}, which provide exactly those alternative phrasings of these notions that are much more nicely compatible with the notion of contents due to a computationally effective formulation using finite unions. In that way, the results from this section tie into the comments made in the introduction that the notions and proofs produced in \cite{ADR2012} through the finitary perspective of proof mining provide the right notions to simultaneously see the results from the light of the theory of contents (and thus, as mentioned before, highlighting the naturalness of the theory of contents as a basis for proof mining in probability theory).\\

We now first note that one direction of that equivalence can be easily witnessed in the system $\mathcal{F}^\omega[\PP,\mathrm{Integral},X]$ discussed previously.

\begin{theorem}\label{thrm:au_to_asfu}
The system $\mathcal{F}^\omega[\PP,\mathrm{Integral},X]$ proves that if $X$ converges almost uniformly with respect to finite unions, then $X$ converges almost uniformly metastable pointwisely.
\end{theorem}
\begin{proof}
We reason in $\mathcal{F}^\omega[\PP,\mathrm{Integral},X]$. Let $k$, $a$ and $F$ be given. Using that $X$ converges almost uniformly w.r.t.\ finite unions, there exists a $b$ such that 
\[
\PP\left(\bigcup_{i=b}^c\bigcup_{j=b}^c P(a,i,j)^c\right)\le 2^{-k}
\]
for all $c$. Now, we in particular have
\[
\bigcap_{m=0}^b\bigcup_{i=m}^{F(m)}\bigcup_{j=m}^{F(m)} P(a,i,j)^c\subseteq \bigcup_{i=b}^{F(b)}\bigcup_{j=b}^{F(b)} P(a,i,j)^c
\]
and therefore 
\[
\PP\left(\bigcap_{m=0}^b\bigcup_{i=m}^{F(m)}\bigcup_{j=m}^{F(m)} P(a,i,j)^c\right)\leq\PP\left(\bigcup_{i=b}^{F(b)}\bigcup_{j=b}^{F(b)} P(a,i,j)^c\right)\leq 2^{-k}
\]
follows from the monotonicity of $\PP$. This yields that $X$ converges almost uniformly metastable pointwisely.
\end{proof}

As mentioned before, a quantitative version of the converse of the above theorem is one of the main results of \cite{ADR2012} and to obtain this, the authors of \cite{ADR2012} mainly utilized a quantitative version of the following property about sequences of events.

\begin{theorem}[Theorem 2.2 of \cite{ADR2012}]
For every sequences of events $(A_n)$, any functional $M: \mathbb{N}^\mathbb{N}\to \NN$ and any $\lambda > \lambda' > 0$, there exists an $n \in \NN$ such that
\[
\PP\left(\bigcap_{m = 0}^{M(F)}\bigcup_{j = m}^{F(m)}A_j\right) < \lambda'\text{ for all }F: \NN \to \NN
\]
implies $\PP(A_n) < \lambda$.
\end{theorem}

We in the following now first discuss how (the proof of) this result can be formalized in our system for probability contents on algebras $\mathcal{F}^\omega[\PP]$, justifying the existence and the uniformity of the quantitative result given in \cite{ADR2012} by the means of our metatheorems. Concretely, we show:

\begin{theorem}\label{thrm:combinatorial_lemma}
The system $\mathcal{F}^\omega[\PP]$ proves:
\[
\forall A^{S(0)}, M^{0(1)}, u^0, v^0>_0 u\exists n^0\left(\forall F^1\left(\PP\left(\bigcap_{m = 0}^{M(F)}\bigcup_{j = m}^{F(m)}A(j)\right) \le_{\mathbb{R}} 2^{-v} \right)\to\PP(A(n)) <_{\mathbb{R}} 2^{-u}\right).
\]
\end{theorem}

In particular, as this theorem of $\mathcal{F}^\omega[\PP]$ has the correct logical form, our main Theorem \ref{thm:metatheorem} on the extraction of uniform computable bounds applies and we thus find that the existence of a computable bound on the existential quantifier on $n$ can be guaranteed to exist a priori and even further, based on our notion of majorizability, it can be guaranteed that this bound is independent of the content space and the sequence of events which matches exactly the properties of the bound explicitly calculated in \cite{ADR2012}.\\

To now demonstrate Theorem \ref{thrm:combinatorial_lemma}, we in particular rely on the following lemma:

\begin{lemma}\label{lem:com_lem_pre}
The system $\mathcal{F}^\omega[\PP]$ proves:
\[
\forall A^{S(0)}, k^0\exists N^0\forall n^0\left(\PP\left(\bigcup_{i=0}^n A(i) \cap \left(\bigcup_{i=0}^N A(i)\right)^c\right) <_\mathbb{R} 2^{-k} \right).
\]
\end{lemma}
\begin{proof}
We reason in $\mathcal{F}^\omega[\PP]$. Let $A^{S(0)}$ and $k^0$ be given. At first, note that Proposition \ref{pro:meta_countable_convergence} implies that
\[
\exists N \forall n \left(n \ge N \to \left\vert \sum_{i=0}^n\PP((A\!\uparrow)(i)) - \sum_{i=0}^N\PP((A\!\uparrow)(i))\right\vert< 2^{-k}\right).\tag{$*$}
\]
Take such an $N$ and let $n$ be arbitrary. If $n<N$, then
\[
\bigcup_{i=0}^n A(i) \cap \left(\bigcup_{i=0}^N A(i)\right)^c=\emptyset
\]
and so by extensionality of $\PP$, we get
\[
\PP\left(\bigcup_{i=0}^n A(i) \cap \left(\bigcup_{i=0}^N A(i)\right)^c\right) = 0
\]
and are done. So suppose $n \ge N$. Then by $(*)$, we get
\[
\left\vert \sum_{i=0}^n\PP((A\!\uparrow)(i)) - \sum_{i=0}^N\PP((A\!\uparrow)(i))\right\vert < 2^{-k}
\]
Since all the $(A\!\uparrow)(i)$ are disjoint (by definition of $A\!\uparrow$) and since we have
\[
\bigcup_{i=0}^j (A\!\uparrow)(i) = \bigcup_{i=0}^j A(i)
\]
for any $j$, we immediately derive
\[
\sum_{i=0}^j\PP((A\!\uparrow)(i)) = \PP\left(\bigcup_{i=0}^j A(i) \right)
\]
for any $j$ by finite additivity and extensionality of $\PP$. Thus, we in particular have 
\[
\left\vert \PP\left(\bigcup_{i=0}^n A(i) \right) - \PP\left(\bigcup_{i=0}^N A(i) \right)\right\vert< 2^{-k}
\]
and since $n\geq N$ implies
\[
\bigcup_{i=0}^N A(i) \subseteq \bigcup_{i=0}^n A(i),
\]
we obtain
\[
\PP\left(\bigcup_{i=0}^n A(i) \cap \left(\bigcup_{i=0}^N A(i)\right)^c\right) = \PP\left(\bigcup_{i=0}^n A(i) \right) - \PP\left(\bigcup_{i=0}^N A(i) \right)< 2^{-k}
\]
by Proposition \ref{prop:properties_of_P}.
\end{proof}

With that lemma, we are now in the position for a formal proof of the main combinatorial theorem from \cite{ADR2012}:

\begin{proof}[Proof of Theorem \ref{thrm:combinatorial_lemma}]
Let $A^{S(0)}$, $M^{0(1)}$, $u^0$ and $v^0$ with $v>u$ be given and suppose
\[
\forall F^1 \left(\PP\left(\bigcap_{m = 0}^{M(F)}\bigcup_{j = m}^{F(m)}A(j)\right) \le 2^{-v} \right).
\]
So, by the previous Lemma \ref{lem:com_lem_pre} applied to the sequence of events $f_m^{S(0)}$ defined by $f_m(k) = A(k+m)$, we have 
\[
\forall m\exists N \forall n\left(\PP\left(\bigcup_{i=m}^{n+m} A(i) \cap \left(\bigcup_{i=m}^{N+m} A(i)\right)^c\right) < \frac{2^{-u}- 2^{-v}}{2^{m+1}}\right)
\]
and so in particular
\[
\forall m\exists N\geq m \forall n\geq m\left(\PP\left(\bigcup_{i=m}^{n} A(i) \cap \left(\bigcup_{i=m}^{N} A(i)\right)^c\right) < \frac{2^{-u}- 2^{-v}}{2^{m+1}}\right).
\]
Thus, using $\mathsf{AC}$ (actually, by switching from $<$ to $\leq$ in the above formulation and manipulating the bound slightly, $\Pi^0_1$-$\mathsf{AC}$ suffices), there exists a functional $F^1$ such that for all $m$ and $n\geq m$:
\[
\PP\left(\bigcup_{i=m}^{n} A(i) \cap \left(\bigcup_{i=m}^{F(m)} A(i)\right)^c\right) < \frac{2^{-u}- 2^{-v}}{2^{m+1}}.
\]
It is now easy to see that for this functional $F$, we have
\begin{align*}
A(M(F))&\subseteq \bigcap_{m = 0}^{M(F)}\bigcup_{i=m}^{M(F)} A(i) \\
&\subseteq\left( \bigcap_{m = 0}^{M(F)}\bigcup_{j = m}^{F(m)}A(j)\right) \cup \bigcup_{m = 0}^{M(F)} \left( \bigcup_{i=m}^{M(F)} A(i) \cap \left(\bigcup_{j = m}^{F(m)}A(j)\right)^c \right)
\end{align*}
and so, by the sub-additivity and monotonicity of $\PP$, we derive 
\[
\PP(A(M(F))) < 2^{-v} + \sum_{m=0}^{M(F)} \frac{2^{-u}- 2^{-v}}{2^{m+1}} < 2^{-u}
\]
and so taking $n:=M(F)$ for this functional $F$ yields the claim.
\end{proof}

Using this combinatory lemma, we can now also prove the converse of Theorem \ref{thrm:au_to_asfu} and thereby exhibit how a quantitative solution for Theorem \ref{thrm:combinatorial_lemma} as obtained in \cite{ADR2012} immediately can be used in conjunction with the proof-theoretic metatheorem established in Theorem \ref{thm:metatheorem} to derive a quantitative version of Egorov's theorem in the sense of the above notions incorporating finite unions as presented in Theorem 3.1 of \cite{ADR2012} and in particular guarantees the observed uniformities of the rates a priori.

\begin{theorem}\label{thrm:au_to_aump}
The system $\mathcal{F}^\omega[\PP,\mathrm{Integral},X]$ proves that if $X$ converges almost uniformly metastable pointwisely, then $X$ converges almost uniformly with respect to finite unions.
\end{theorem}
\begin{proof}
Suppose that $X$ does not converge almost uniformly with respect to finite unions, i.e.\ that we have $k$ and $a$ such that
\[
\forall m\exists g\left( \PP\left(\bigcup_{i=m}^g\bigcup_{j=m}^g P(a,i,j)^c\right)> 2^{-k}\right).
\]
Using $\QFAC$ (after suitably prenexing the hidden quantifiers), we get a functional $G$ such that
\[
\forall m\left(\PP\left(\bigcup_{i=m}^{G(m)}\bigcup_{j=m}^{G(m)} P(a,i,j)^c\right)> 2^{-k}\right).
\]
For a contradiction, suppose now that $X$ converges almost uniformly metastable pointwisely. By instantiating the corresponding notion with $k+1$ and $a$, we get
\[
\forall F\exists b\left( \PP\left(\bigcap_{m=0}^b\bigcup_{i=m}^{F(m)}\bigcup_{j=m}^{F(m)}P(a,i,j)^c\right)\leq 2^{-(k+1)}\right).
\]
Thus, by $\QFAC$ (again after suitably prenexing), we get a functional $M$ such that
\[
\forall F\left( \PP\left(\bigcap_{m=0}^{M(F)}\bigcup_{i=m}^{F(m)}\bigcup_{j=m}^{F(m)}P(a,i,j)^c\right)\leq 2^{-(k+1)}\right).
\]
Defining $M'(F)=M(\lambda n.\tilde{G}(F(n)))$ where $\tilde{G}(n)=\max_{m \le n} G(m)$, we get 
\[
\forall F\left( \PP\left(\bigcap_{m=0}^{M'(F)}\bigcup_{i =m}^{\tilde{G}(F(m))}\bigcup_{j=m}^{\tilde{G}(F(m))}P(a,i,j)^c\right)\leq 2^{-(k+1)}\right).
\]
We now define a sequence of events $A$ via
\[
A(n):=\bigcup_{i=n}^{G(n)}\bigcup_{j=n}^{G(n)}P(a,i,j)^c.
\]
Then we have
\[
\bigcup_{n=m}^{F(m)}A(n) = \bigcup_{n=m}^{F(m)}\bigcup_{i=n}^{G(n)}\bigcup_{j=n}^{G(n)}P(a,i,j)^c=\bigcup_{i=m}^{\tilde{G}(F(m))}\bigcup_{j=m}^{\tilde{G}(F(m))}P(a,i,j)^c
\]
for all $m$ and $F$ and therefore this implies 
\[
\forall F\left( \bigcap_{m=0}^{M'(F)}\bigcup_{n=m}^{F(m)}A(n) = \bigcap_{m=0}^{M'(F)}\bigcup_{i=m}^{\tilde{G}(F(m))}\bigcup_{j=m}^{\tilde{G}(F(m))}P(a,i,j)^c\right).
\]
By the extensionality of $\PP$, we get
\[
\forall F\left( \PP\left(\bigcap_{m=0}^{M'(F)}\bigcup_{n=m}^{F(m)}A(n)\right) \leq 2^{-(k+1)}\right),
\]
which yields, by Theorem \ref{thrm:combinatorial_lemma}, that there exists an $m$ such that
\[
\PP(A(m)) = \PP\left(\bigcup_{i=m}^{G(m)}\bigcup_{j=m}^{G(m)} P(a,i,j)^c\right) < 2^{-k}
\]
which is a contradiction.
\end{proof}

As a last formal elucidation of some of the results from \cite{ADR2012}, we turn to Theorem 3.2 therein, where the authors provide a quantitative version for a special case of the dominated convergence theorem, strengthening preceding results from Tao \cite{Tao2008}. Concretely, they assume for this special case that the random variables are positive and bounded (w.l.o.g.) by $1$ (which immediately yields the ``uniform'' majorizability of the sequence $X$ and thus guarantees the full independence of the rate from $X$ via the preceding metatheorems). The following result now establishes that the corresponding infinitary convergence result can be proved in our system for integrals over probability contents $\mathcal{F}^\omega[\PP,\mathrm{Integral},X]$ and therefore makes it possible to recognize the quantitative results extracted in \cite{ADR2012} as an application of the metatheorems established in this paper.

\begin{theorem}
The system $\mathcal{F}^\omega[\PP,\mathrm{Integral},X]$ proves: if $X$ converges almost uniformly metastable pointwisely and satisfies $\forall n^0,x^\Omega\left(0\leq_\mathbb{R} X(n)(x)\leq_\mathbb{R}1\right)$, it holds that
\[
\forall k^0\exists n^0 \forall i^0,j^0\left(i,j\ge_0 n \to \left \vert\INT{X(i)} - \INT{X(j)}\right\vert\le_\RR 2^{-k}\right).
\]
\end{theorem}
\begin{proof}
Let $k$ be given. Since, by Theorem \ref{thrm:au_to_aump}, $X$ converges almost uniformly w.r.t.\ finite unions, there exists an $n$ such that
\[
\forall m\left( \PP\left(\bigcup_{a=n}^m\bigcup_{b=n}^m P(k+1,a,b)^c\right)\le 2^{-(k+2)}\right).
\]
Take $i,j\geq n$ and define $m=\max\{i,j\}$ as well as
\[
A:=\bigcup_{a=n}^m\bigcup_{b=n}^m P(k+1,a,b)^c.
\]
Similar as in the proof of item (4) from Lemma \ref{lem:intProp}, we have
\begin{align*}
\left\vert\INT{X(i)} - \INT{X(j)}\right\vert&\le\INT{\left\vert X(i)-X(j)\right\vert}\\
&=\INT{\left\vert X(i)-X(j)\right\vert\chi_A} + \INT{\left\vert X(i)-X(j)\right\vert\chi_{A^c}}.   
\end{align*}
As we have $\left\vert X(i)(x)-X(j)(x)\right\vert\le 2$ for all $x$, we get  
\[
\INT{\left\vert X(i)-X(j)\right\vert\chi_A}\le \INT{2\chi_A}= 2\PP(A)\le 2^{-(k+1)}.
\]
On the other hand, $x\in A^c$ yields 
\[
x\in \bigcap_{a=n}^m\bigcap_{b=n}^m P(k+1,a,b)
\]
which in particular gives, by definition of $m$, that $x\in P(k+1,i,j)$. From the definition of $P$, this in particular implies that
\[
\left \vert X(i)(x) -X(j)(x)\right \vert \le 2^{-(k+1)}
\]
and so we have $\left\vert X(i)(x) -X(j)(x)\right\vert\chi_{A^c}(x) \le 2^{-(k+1)}$ with yields
\[
\INT{\left\vert X(i) -X(j)\right\vert\chi_{A^c}}\le 2^{-(k+1)}
\]
and we are done.
\end{proof}

\section{Proof-theoretic transfer principles}\label{sec:transfer}

In this last section, we present how our systems and metatheorems allow for the proof of a general type of result, which we call a proof-theoretic transfer principle, that allows one to transfer quantitative information on implications between modes of convergence of real numbers to corresponding quantitative information on implications between analogous modes of convergence for bounded random variables. In particular, as this type of reasoning is very common in the literature on the convergence of various iterations of random variables (see e.g.\ \cite{Kol1930} among many others), this transfer principle allows for a logical explanation of the strategy of providing a proof-theoretic analysis of such results in practice by mainly analysing the underlying result on real numbers and then lifting this result together with some (often) simple modifications to random variables.\\

Concretely, to allow for a discussion of general modes of convergence for real numbers and random variables, we consider the following abstract formal setup: We throughout this section fix two $\Pi_3$-formulas 
\[
P(x^{1(0)},\underline{p}^{\underline{\sigma}})=\forall a^0\exists b^0\forall c^0 P_0(a,b,c,\widehat{x},\underline{p})
\]
and
\[
Q(x^{1(0)},\underline{p}^{\underline{\sigma}})=\forall u^0\exists v^0\forall w^0 Q_0(u,v,w,\widehat{x},\underline{p})
\]
where $P_0$ and $Q_0$ are quantifier-free formulas which only have the indicated variables free and where we write $\widehat{x}:=\lambda n.\widehat{x(n)}$. We understand $P$ and $Q$ as abstract representations of modes of convergence for the parameter sequence $x$ of real numbers with parameters $\underline{p}$. This perspective of $x$ representing a sequence of real numbers is also why we have used $\widehat{x}$ in the internal predicates $P_0$ and $Q_0$ inducing $P$ and $Q$.\\

For an example, we may take 
\[
P_0(a,b,c,\widehat{x}):= \forall i^0,j^0\left( b\leq_0 i,j\land i,j\leq_0c\to |\widehat{x(i)} - \widehat{x(j)}|\in [0,2^{-a}]\right),\tag{$**$}
\]
using the previous intensional intervals (where the above statement can thus be regarded as a quantifier-free statement). In that case, $P$ represents the usual Cauchy-property for $x$.\\

To allow for a discussion of these modes applied to random variables, we extend the system $\mathcal{F}^\omega[\PP]$ with four further constants
\[
X^{1(\Omega)(0)},\; P^{S\underline{\sigma}^t(0)(0)(0)},\; Q^{S\underline{\sigma}^t(0)(0)(0)},\; \tau^{0(0)},
\]
together with the axioms
\begin{gather*}
\forall \underline{p}^{\underline{\sigma}},a^0,b^0,c^0,z^\Omega(z \in P(a,b,c,\underline{p}) \leftrightarrow P_0(a,b,c,\lambda n.\widehat{X(n)(z)},\underline{p})),\\
\forall \underline{p}^{\underline{\sigma}},a^0,b^0,c^0,z^\Omega(z \in Q(a,b,c,\underline{p}) \leftrightarrow Q_0(a,b,c,\lambda n.\widehat{X(n)(z)},\underline{p})),\\
\forall n^0, z^\Omega(\tau(n) \le_0 \tau(n+1)\land \tau(n) \ge_\RR \vert X(n)(z)\vert ),
\end{gather*}
specifying that the properties $P_0$ and $Q_0$ (inducing the predicates $P$ and $Q$) induce measurable sets pointwisely relative to the sequence of random variables\footnote{If we assume that $X$ is weakly Borel-measurable in the context of the above example $(**)$, the corresponding point sets $P,Q$ as above indeed are measurable.} specified by $X$ and that these random variables are all bounded via a suitable monotone sequence of bounds (i.e.\ that $X$ as a constant is majorized by $\tau$). It is clear that Theorem \ref{thm:metatheorem} extends to this system, which we denote by $\mathcal{U}^\omega$, as all constants are majorizable and since the new axioms are purely universal.\\

In this extended language, we can the provide a formula that represents the property $P$ if suitably lifted to the sequence of random variables represented by $X$:

\begin{definition}
We say that $X$ satisfies $P$ almost uniformly, and write $P(X)$ a.u., if 
\[
\forall \underline{p}^{\underline{\sigma}},k^0,a^0\exists b^0\forall c^0\left(\PP\left(P(a,b,c,\underline{p})^c\right)\le_\RR 2^{-k}\right).
\]
Similarly, we define $Q(X)$ a.u.
\end{definition}

If we consider the previous example for $P_0$ given in $(**)$, then by formulating $P(X)$ a.u.\ in this case we recover the notion of almost uniform convergence with respect to finite unions as given in Definition \ref{def:au} (i.e.\ the variant of almost uniform convergence implicitly considered in \cite{ADR2012}).\\

We now turn to our main result that provides a relationship between statements of the form 
\[
\forall \underline{p}^{\underline{\sigma}},x^{1(0)} (P(x,\underline{p}) \to Q(x,\underline{p}))
\]
and statements of the form
\[
P(X)\text{ a.u.} \to Q(X)\text{ a.u.}
\]
and which thereby not only establishes an upgrade-type theorem from relations between modes of convergence for sequences of reals to sequences of random variables but also allows for a transfer of the computational information obtainable for the implication in the premise to the implication in the conclusion.

\begin{theorem}\label{thrm:transfer}
Provably in $\mathcal{U}^\omega$, given functionals $V,A,C$ such that
\[
\forall x,\underbrace{\underline{p},x^*,B,u,w}_{\omega}\left( x^*\gtrsim x\land P_0(A\omega,B(A\omega),C\omega,\widehat{x},\underline{p})\to Q_0(u,V\underline{p}x^*Bu,w,\widehat{x},\underline{p})\right),
\]
we can construct $V',A',C'$ such that
\[
\forall \underbrace{\underline{p},B,k,u,w}_{\alpha}\left( \PP(P(A'\alpha,Bk(A'\alpha),C'\alpha,\underline{p})^c)\leq 2^{-k}\to \PP(Q(u,V'\underline{p}Bku,w,\underline{p})^c)\leq 2^{-k}\right).
\]
\end{theorem}
\begin{proof}
Given such $V,A,C$ and $\alpha=(\underline{p},B,k,u,w)$, we define 
\begin{gather*}
A'\alpha:=A\underline{p}\overline{\tau}(Bk)uw,\\
C'\alpha:=C\underline{p}\overline{\tau}(Bk)uw,\\
V'\underline{p}Bku:=V\underline{p}\overline{\tau}(Bk)u.
\end{gather*}
for $\overline{\tau}=\lambda n.(\tau(n)(0)+1)$. Let $z$ be arbitrary with $z\in Q(u,V'\underline{p}Bku,w,\underline{p})^c$. By the axioms of $\mathcal{U}^\omega$ and the definition of $V'$, we have
\begin{align*}
z\in Q(u,V'\underline{p}Bku,w,\underline{p})^c&\leftrightarrow z\in Q(u,V\underline{p}\overline{\tau}(Bk)u,w,\underline{p})^c\\
&\leftrightarrow \neg Q_0(u,V\underline{p}\overline{\tau}(Bk)u,w,\lambda n.\widehat{X(n)(z)},\underline{p})
\end{align*}
and the latter implies 
\[
\neg P_0(A\underline{p}\overline{\tau}(Bk)uw,Bk(A\underline{p}\overline{\tau}(Bk)uw),C\underline{p}\overline{\tau}(Bk)uw,\lambda n.\widehat{X(n)(z)},\underline{p})
\]
using the assumptions on $V,A,C$ and that $\tau(n)\geq \vert X(n)(z)\vert$ and so $\lambda i.(\tau(n)(0)+1)\gtrsim \widehat{X(n)(z)}$ for any $z$. This in turn is by definition of $A',V',C'$ equivalent to
\[
\neg P_0(A'\alpha,Bk(A'\alpha),C'\alpha,\lambda n.\widehat{X(n)(z)},\underline{p})
\]
and thus to
\[
z\in P(A'\alpha,Bk(A'\alpha),C'\alpha,\underline{p})^c.
\]
Thus, we have
\[
Q(u,V'\underline{p}Bku,w,\underline{p})^c\subseteq P(A'\alpha,Bk(A'\alpha),C'\alpha,\underline{p})^c
\]
as $z$ above was arbitrary and therefore, we get
\[
\PP(Q(u,V'\underline{p}Bku,w,\underline{p})^c)\leq \PP(P(A'\alpha,Bk(A'\alpha),C'\alpha,\underline{p})^c)
\]
by the monotonicity of $\PP$. This yields the claim.
\end{proof}

This result, while on a first look rather technical and abstract, has a very concrete use recently observed in applications by the first author \cite{Ner2023} and to illustrate this, we will now shortly discuss the extend of the above result and its use in mathematics:
\begin{enumerate}
\item Observe that the conclusion of Theorem \ref{thrm:transfer} is just a witnessed version of the Dialectica interpretation of
\[
P(X)\text{ a.u.} \to Q(X)\text{ a.u.}\tag{$+$}
\]
and therefore, under $\AC$, this witnessed Dialectica interpretation in particular implies $(+)$. In that way, whenever the premise of Theorem \ref{thrm:transfer} is established in $\mathcal{U}^\omega$, one immediately obtains the truth of $(+)$ and so the above result allows for a lift from a (quantitative) result on real numbers to a true result for random variables. Furthermore, another main benefit of the conclusion of Theorem \ref{thrm:transfer} is that it allows for the extraction of quantitative information in the sense that the functional $V'$ provides a transformation of a rate for the premise $P(X)$ a.u. into a rate for the conclusion $Q(X)$ a.u. Even further, as $V'$ can be constructed from $V$, this transformation of rates can be directly inferred from the transformation of rates $V$ of the presumed result for real numbers.
\item The premise of Theorem \ref{thrm:transfer} is essentially the Dialectica interpretation of the statement 
\[
\forall \underline{p}^{\underline{\sigma}},x^{1(0)} (P(x,\underline{p}) \to Q(x,\underline{p}))\tag{$++$}
\]
in the sense that the functionals $V,A,C$ represent realizers for this interpretation with the additional assumption that these realizers are suitably uniform, depending only on a majorant of the sequence $x^{1(0)}$. Although one can construct examples where such a uniformity of the realizers is not the case, in practice, for many theorems of the form $(++)$ that have a semi-constructive proof, such uniform realizers can be given. In particular, this is true for the forthcoming work by the first author on Kronecker's Lemma \cite{Ner2023} and an upcoming work by Oliva and Arthan on quantitative stochastic optimisation \cite{OA2023}. In both these cases, one uses reasoning about sequences of real numbers to obtain the analogous result about sequences of random variables and a computational interpretation can be given to this line of reasoning. Theorem \ref{thrm:transfer} then in particular provides an abstract generalization of this procedure and explains how this reasoning is substantiated by logical results.
\end{enumerate}

Lastly, in the following remark we discuss a counterexample illustrating the necessity of the majorizability of the sequence of random variables in Theorem \ref{thrm:transfer}:

\begin{remark}
For the above transfer principle to hold, the assumption of the boundedness of the sequence of random variables is necessary as the following example shows: Take $\Omega:= \NN$ and let $S$ be the collection of all finite and co-finite subsets of $\NN$, i.e.\ 
\[
S:= \{A \subseteq \NN: A \text{ is finite or } A^c \text{ is finite}\}.
\]
Furthermore, define the measure $\PP$ by $\PP(A) = 0$ if $A$ is finite and $\PP(A) = 1$ if $A^c$ is finite, for all $A \in S$. Now, we consider the two properties 
\[
P(x)=P_0(\widehat{x})\equiv 0=0\text{ and }Q(x)\equiv \exists n \forall mQ_0(n,m,\widehat{x})\equiv \exists n\forall m(n\geq_{\mathbb{Q}} [\widehat{x_0}](m))
\]
for a sequence $x=(x_n)$ of real numbers. Clearly, both $P$ and $Q$ are $\Pi_3^0$-formulas and are trivially true for all sequences $x$. Therefore also $P(x)\to Q(x)$ is trivially true. Further, we can easily give $V,A,C$ that satisfy the assumptions of Theorem \ref{thrm:transfer}. Now, consider 
\[
X_n:\NN \to \NN,\quad k\mapsto k
\]
for each $n$. Then the set $Q(n,m)$ corresponding to $Q_0$ is just 
\[
Q(n,m)=\{k\in\mathbb{N}\mid Q_0(n,m,\lambda l.\widehat{X(l)(k)}\}=\{k\in\mathbb{N}\mid n\geq_{\mathbb{Q}} [\widehat{X(0)(k)}](m)\}=\{k\in\mathbb{N}\mid n\geq k\}
\]
which belongs to $S$ as it is finite. $P_0$ is just represented by the full set $\mathbb{N}$. Therefore, $X$ satisfies $P$ almost uniformly and does not satisfy $Q$ almost uniformly as any $Q(n,m)^c$ has measure $1$.
\end{remark}

\noindent
{\bf Acknowledgments:} The first author was partially supported by the EPSRC Centre for Doctoral Training in Digital Entertainment (EP/L016540/1). The second author was supported by the `Deutsche Forschungs\-gemein\-schaft' Project DFG KO 1737/6-2. Both authors want to thank Thomas Powell, Ulrich Kohlenbach, Henry Towsner and Jos\'e Iovino for insightful discussions on the topics of this paper.

\noindent
{\bf Declarations of interest:} none

\bibliographystyle{plain}
\bibliography{ref}

\end{document}